%% file: hhftc.tex
\documentclass[10pt,a4paper]{article}

\let\emph\undefined
\newcommand{\emph}[1]{{\slshape #1}}

\usepackage{amssymb}
\usepackage{amsmath}
\usepackage{amsfonts}
\usepackage{bbm}
\usepackage{amsthm}
\usepackage{mathrsfs}
\usepackage[hidelinks]{hyperref}
\usepackage{color}
\usepackage[margin=2cm]{geometry}
\usepackage[all,cmtip]{xy}
\usepackage[utf8]{inputenc}
\usepackage{graphicx}
\usepackage{varwidth}

\usepackage[font={it}]{caption}

\usepackage{needspace}
\newcommand{\spaceplease}{\needspace{5\baselineskip}}

\usepackage{upgreek}
\usepackage{rotating}
\usepackage{tikz}
\usetikzlibrary{matrix,arrows,decorations.pathmorphing,shapes.geometric}
\usepackage{tikz-cd}
\usepackage{mathtools}
\mathtoolsset{showonlyrefs}
\usepackage{epstopdf}

\newtheoremstyle{mytheorem}
{\topsep}
{\topsep}
{\slshape}
{0pt}
{\bfseries}
{.}
{ }
{\thmname{#1}\thmnumber{ #2}\thmnote{ {\normalfont\slshape(#3)}}}

 \newtheoremstyle{mydefinition}
 {\topsep}
 {\topsep}
 {\normalfont}
 {0pt}
 {\bfseries}
 {.}
 { }
 {\thmname{#1}\thmnumber{ #2}\thmnote{ {\normalfont\slshape(#3)}}}

\theoremstyle{mytheorem}
\newtheorem{theorem}{Theorem}[section]

\makeatletter
\newtheorem*{rep@theorem}{\rep@title}
\newcommand{\newreptheorem}[2]{%
	\newenvironment{rep#1}[1]{%
		\def\rep@title{#2 \ref{##1}}%
		\begin{rep@theorem}}%
		{\end{rep@theorem}}}
\makeatother

\newreptheorem{theorem}{Theorem}
\newreptheorem{proposition}{Proposition}
\newreptheorem{corollary}{Corollary}

\newtheorem{lemma}[theorem]{Lemma}
\newtheorem{proposition}[theorem]{Proposition}
\newtheorem{corollary}[theorem]{Corollary}
\theoremstyle{mydefinition}
\newtheorem{definition}[theorem]{Definition}
\newtheorem{exx}[theorem]{Example}
\newenvironment{example}
{\pushQED{\qed}\exx}
{\popQED\endexx}
\newtheorem{remm}[theorem]{Remark}
\newenvironment{remark}
{\pushQED{\qed}\remm}
{\popQED\endremm}
\numberwithin{equation}{section}
\usepackage{enumitem}

\newenvironment{xenumerate}{\begin{enumerate}[topsep=2pt,parsep=2pt,partopsep=2pt,itemsep=0pt,label={\normalfont(\arabic*)}]\itemsep0pt}{\end{enumerate}}

\DeclareMathSymbol{\Phiit}{\mathalpha}{letters}{"08}\let\Phi\undefined\newcommand{\Phi}{\Phiit}
\DeclareMathSymbol{\Psiit}{\mathalpha}{letters}{"09}\let\Psi\undefined\newcommand{\Psi}{\Psiit}
\DeclareMathSymbol{\Sigmait}{\mathalpha}{letters}{"06}\let\Sigma\undefined\newcommand{\Sigma}{\Sigmait}
\DeclareMathSymbol{\Xiit}{\mathalpha}{letters}{"04}
\DeclareMathSymbol{\Lambdait}{\mathalpha}{letters}{"03}\let\Lambda\undefined\newcommand{\Lambda}{\Lambdait}
\DeclareMathSymbol{\Piit}{\mathalpha}{letters}{"05}\let\Pi\undefined\newcommand{\Pi}{\Piit}
\DeclareMathSymbol{\Gammait}{\mathalpha}{letters}{"00}\let\Gamma\undefined\newcommand{\Gamma}{\Gammait}
\DeclareMathSymbol{\Omegait}{\mathalpha}{letters}{"0A}\let\Omega\undefined\newcommand{\Omega}{\Omegait}
\DeclareMathSymbol{\Upsilonit}{\mathalpha}{letters}{"07}\let\Upsilon\undefined\newcommand{\Upsilon}{\Upilonit}
\DeclareMathSymbol{\Thetait}{\mathalpha}{letters}{"02}\let\Theta\undefined\newcommand{\Theta}{\Thetait}

\let\e\undefined\newcommand{\e}{^\mathrm{e}}
\def\Hom{\mathrm{Hom}}

\def\id{\mathrm{id}}
\def\SL{\operatorname{SL}}

\let\to\undefined\newcommand{\to}{\longrightarrow}
\let\mapsto\undefined\newcommand{\mapsto}{\longmapsto}
\newcommand{\catf}[1]{\mathsf{#1}}
\newcommand{\Proj}{\operatorname{\catf{Proj}}}
\newcommand{\Mod}{\catf{Mod}}

\newcommand{\Loop}{\mathcal{L}}

\def\op{\mathrm{op}}

\def\Ch{\catf{Ch}_k}

\def\ChA{\mathsf{Ch}_R}

\newcommand{\hocolimsub}[1]{\underset{#1}{\operatorname{hocolim}}\,}
\def\PBun{\catf{PBun}}
\def\nE{\bar E}

\let\P\undefined
\newcommand{\P}{\operatorname{P}}
\newcommand{\Vect}{\catf{Vect}}
\newcommand{\cof}{\mathsf{Q}}

\newcommand{\cat}[1]{\mathcal{#1}}
\newcommand{\sign}{\operatorname{sign}  }

\newcommand{\Cat}{\catf{Cat}}

\newcommand{\Alg}{\catf{Alg}}

\newcommand{\PBr}{\mathsf{PBr}}
\let\O\undefined
\newcommand{\O}{\mathcal{O}}
\let\Bar\undefined
\newcommand{\Bar}{\operatorname{Bar}}
\newcommand{\lint}{\int_\mathbb{L}}
\newcommand{\fint}{\int_\text{\normalfont f}}
\newcommand{\flint}{\int_{\text{\normalfont f}\mathbb{L}}}

\newcommand{\sVect}{\catf{sVect}}
\newcommand{\Aut}{\operatorname{Aut}}

\definecolor{Blue}  {rgb} {0.282352,0.239215,0.803921}
\definecolor{Green} {rgb} {0.133333,0.545098,0.133333}
\definecolor{Red}   {rgb} {0.803921,0.000000,0.000000}
\definecolor{Violet}{rgb} {0.580392,0.000000,0.827450}

\newcounter{jfc}

\begin{document}

	\begin{flushright}
		\small
		{\sf [ZMP-HH/19-17]} \\
		\textsf{Hamburger Beiträge zur Mathematik Nr.~802}\\
		\textsf{October 2019}
	\end{flushright}
	
	\vspace{15mm}
	
	\begin{center}
		\textbf{\LARGE{The Hochschild Complex of a Finite Tensor Category
		}}\\
		\vspace{1cm}
		{\large Christoph Schweigert and  Lukas Woike }
		
		\vspace{5mm}
		
	\normalsize
		{\slshape Fachbereich Mathematik\\ Universit\"at Hamburg\\
			Bereich Algebra und Zahlentheorie\\
			Bundesstra\ss e 55\\ D\,--\,20\,146\, Hamburg }
	\end{center}
	\vspace{0.3cm}
	\begin{abstract}
		\noindent Modular functors, i.e.\ 
		consistent systems of
		projective representations of mapping class groups of surfaces, have been constructed
		for non-semisimple modular categories already decades ago. Concepts from homological algebra have not been used in this construction although it is an obvious
		question  how they should enter in the non-semisimple case. In the
		present paper, we elucidate the interplay between the structures from
		topological field theory and from homological algebra by constructing
		a \emph{homotopy coherent} projective action of the mapping
		class group $\SL(2,\mathbb{Z})$ of the torus on the Hochschild complex of a
		modular category. This is a further step towards understanding
		the Hochschild complex of a modular category as a
		\emph{differential graded conformal block} for the torus. Moreover, we describe a differential graded version of the Verlinde algebra. 
	\end{abstract}

\tableofcontents
		
\newpage

\section{Introduction and outlook}
In this article, we prove several results for the Hochschild complex of a class of linear categories relevant in the representation theory of finite groups (or, more generally, finite-dimensional 
Hopf algebras) and for the construction of topological field theories and modular functors, namely \emph{finite tensor categories} as introduced in \cite{etinghofostrik}. These are linear Abelian monoidal categories with an exact tensor product which satisfy finiteness conditions and are rigid (we will additionally assume throughout that the base field is algebraically closed).
Finite tensor categories are \emph{not} assumed to be semisimple.
By the Hochschild complex of a finite tensor category $\cat{C}$ (or more generally any linear category) we understand the differential graded vector space given by 
the homotopy coend
\begin{align} \lint^{X \in \Proj \cat{C}} \cat{C}(X,X)  \label{vsontoruseqn}
\end{align}
over the endomorphism spaces $\cat{C}(X,X)$ of projective objects $X$ in $\cat{C}$ (depending on the terminology that one prefers, one might also call this the Hochschild complex of the category $\Proj \cat{C} \subset \cat{C}$ of projective objects). 
Homotopy coends and Hochschild homology are recalled in 
Section~\ref{secderivedenrichedcoends}.
If $\cat{C}$ is written as finite-dimensional modules over a finite-dimensional algebra $A$, then \eqref{vsontoruseqn} is equivalent to the Hochschild complex of $A$ \cite{mcarthy,keller}.\\

While we also investigate the Hochschild complex for general finite tensor categories, our main results apply to a certain class of finite tensor categories particularly relevant in quantum topology, especially in the study of Hopf algebras and vertex operator algebras, namely \emph{modular categories}  \cite{turaev,huang,egno}, i.e.\ finite tensor categories with braiding and ribbon structure such that the braiding is non-degenerate meaning that the only objects that trivially double braid with all other objects are finite direct sums of the unit. Note that this notion of modularity does not include semisimplicity. 

For a semisimple modular category $\cat{C}$, the Hochschild complex is equivalent to its zeroth homology, i.e.\  the vector space $\int^{X \in \cat{C}}\cat{C}(X,X)$. 
This vector space arises by evaluation of the Reshetikhin-Turaev topological field theory for $\cat{C}$ on the torus;
 we refer to  \cite{rt1,rt2,turaev} for the Reshetikhin-Turaev construction 
 and to \cite{BDSPV15} for the classification of 3-2-1-dimensional topological field theories by semisimple modular categories.
 This topological perspective on semisimple modular categories has two immediate consequences:
 \begin{enumerate}[label={\normalfont(\arabic*)}]
 	
 	\item Multiplicative structure: If we denote by $P : \mathbb{S}^1 \sqcup \mathbb{S}^1 \to   \mathbb{S}^1$ the pair of pants bordism, then we can evaluate the topological field theory associated to $\cat{C}$ on the bordism $P \times \mathbb{S}^1 : \mathbb{T}^2 \sqcup \mathbb{T}^2 \to \mathbb{T}^2$. This yields an associative  multiplication on $\int^{X \in \cat{C}} \cat{C}(X,X)$ which is induced by the tensor product of $\cat{C}$. The braiding of $\cat{C}$ ensures that the multiplication is commutative. The vector space $\int^{X \in \cat{C}} \cat{C}(X,X)$ with this multiplication is referred to as the \emph{Verlinde algebra of $\cat{C}$}.\label{pointverlinde}
 	
 	\item Mapping class group action: By being the value of a topological field theory on the torus,
 	$\int^{X\in\cat{C}}\cat{C}(X,X)$ carries an action of the mapping class group $\SL(2,\mathbb{Z})$ of the torus. 
 	As a result of the framing anomaly, this action will generally be only projective (the projectiveness can be built into the 3-2-1-dimensional bordism category, see however Remark~\ref{remproj}). The mapping class group action is not through algebra automorphisms of the Verlinde algebra (except for trivial cases). \label{pointmcg}
 	
 	\end{enumerate}
For a \emph{non-semisimple} modular category $\cat{C}$,
an equally satisfactory topological understanding of the Hochschild complex $\lint^{X \in \Proj \cat{C}} \cat{C}(X,X)$ is not available.
Still, we may ask whether the chain complex $\lint^{X \in \Proj \cat{C}} \cat{C}(X,X)$ carries a multiplicative structure and a mapping class group action generalizing the ones encountered in the semisimple case.  This paper answers these questions affirmatively. The difficulty in providing the needed generalizations is distributed unevenly between points~\ref{pointverlinde} and~\ref{pointmcg}.
Generalizing the multiplicative structure is relatively straightforward while establishing the mapping class group action is significantly more involved. \\

Let us state the results in detail: The commutative multiplication of the Verlinde algebra is replaced by an $E_2$-commutative multiplication, i.e.\ a multiplication whose commutativity behavior is controlled by the braid group:

\begin{repproposition}{propdgva}
	For every braided finite tensor category $\cat{C}$,
	the Hochschild complex $ \lint^{X \in \Proj \cat{C}} \cat{C}(X,X)$ is naturally a non-unital $E_2$-algebra in differential graded vector spaces. \end{repproposition}
This statement is established in Section~\ref{secdgva}. Besides proving the result abstractly, we concretely give the $E_2$-multiplication at the chain level (Corollary~\ref{core2atchainlevel}).

The strategies used to obtain a differential graded version of the Verlinde algebra and its  motivation by topological field theory
can be used to obtain,
for any finite group $G$, a candidate for Hochschild chains on a
braided $G$-crossed monoidal category $\cat{C} = \bigoplus_{g\in G} \cat{C}_g$
in the sense of Turaev
\cite{turaevhqft}, see \cite{galindo} for a slightly different definition that we will adopt.
We prove that this equivariant Hochschild complex carries an action of an operad built from 
Hurwitz spaces (Proposition~\ref{proplittlebundlesalgebra}); we refer to 
Section~\ref{secequiv}
for the details. \\

As our main result, we lift the mapping class group action from point~\ref{pointmcg} to a homotopy coherent framework:

\begin{reptheorem}{thmsl2z}
	The Hochschild complex $\lint^{X \in \Proj \cat{C}} \cat{C}(X,X)$ of a modular category $\cat{C}$
	carries a canonical homotopy coherent projective action of the mapping class group $\SL(2,\mathbb{Z})$ of the torus.
\end{reptheorem}
The proof of 
Theorem~\ref{thmsl2z}
 relies on a detailed investigation of the Hochschild complex of a finite tensor category $\cat{C}$:
As a key tool, we introduce in Section~\ref{sectracesclass} a specific projective resolution $\flint^{X \in \Proj \cat{C}}X \otimes  X^\vee$ of the canonical coend $\mathbb{F}=\int^{X \in \cat{C}} X \otimes X^\vee$ of a finite tensor category from \cite{lubacmp,luba,kl} and prove that we may express  the Hochschild chains of $\cat{C}$ up to equivalence by
\begin{align}
\lint^{X \in \Proj \cat{C}} \cat{C}(X,X) \simeq \cat{C}\left(I, \flint^{X \in \Proj \cat{C}}  X \otimes X^\vee \right) \ ,              \label{eqncorresocoend2}
\end{align}
see~Corollary~\ref{corresocoend2}.
If $\cat{C}$ is pivotal, we have such an equivalence also for \emph{any} projective resolution of $\mathbb{F}$ (Theo\-rem~\ref{theoderivedcoendviaobject}).
The proof of the latter fact uses the modified trace  on the tensor ideal of projective objects of a pivotal finite tensor category \cite{mtrace}.

Theorem~\ref{thmsl2z} extends previous results in this direction:
By \cite{sz} there is a projective $\SL(2,\mathbb{Z})$-action on the center of a ribbon factorizable Hopf algebra. Inspired by the fact that the center is just the zeroth Hochschild cohomology, in \cite{svea}  a projective $\SL(2,\mathbb{Z})$-action on the Hochschild cohomology of a ribbon factorizable Hopf algebra is constructed. In \cite{shimizu} this is phrased in terms of the representation categories. These actions  exist only on the (co)homology. It is a natural question to ask whether this action
can be described in a canonical way as a \emph{homotopy coherent projective action at the chain level}. Theorem~\ref{thmsl2z} answers this question affirmatively, see Corollary~\ref{corhopfalg} for the Hopf algebraic version.

The projectivity of this action arises naturally from the construction. For this reason, we work consistently with projective actions regardless of whether the action can be made linear by appropriate choices, see also Remark~\ref{remproj}.

Proposition~\ref{propdgva} and
Theorem~\ref{thmsl2z} suggest the interpretation of the complex $\lint^{X \in \Proj \cat{C}} \cat{C}(X,X)$ as a differential graded generalization of the conformal block for the torus in the sense of \cite{baki}.
In fact, this viewpoint very much informs the proof of 
 Theorem~\ref{thmsl2z}:
The complex $\lint^{X \in \Proj \cat{C}} \cat{C}(X,X)$ is expressed in a slightly different way which is inspired by a differential graded version of the factorization axiom for conformal blocks.	
Then we use the projective action of the braid group $B_3$ on three strands (the mapping class group of the torus with a disk removed)
on $\mathbb{F}$ given in \cite{lubamajid,lubacmp,luba}. The group $B_3$ is a central extension of $\SL(2,\mathbb{Z})$, and we prove and use a criterion for the projective action of $B_3$ to descend to a homotopy coherent projective action of $\SL(2,\mathbb{Z})$. 
This can be seen as a homotopy coherent version of the strategy used in  \cite{svea2}, see~Remark~\ref{remarkfactorization}.\\

 The construction of a fully-fledged differential graded modular functor, including a differential graded version of the sewing axioms and homotopy coherent projective actions of the mapping class groups of higher genus surfaces, is beyond the scope of this article.
 On the way to such a construction, the treatment of the torus is more than a special case. Instead, it is a key ingredient and technical prerequisite. The reason for this is the distinguished role played by genus one data in the Lego-Teichmüller game of Bakalov and Kirillov  \cite{bakifm}.
  
\subparagraph{Conventions.} Throughout this text, we will work over an algebraically closed field $k$ which is \emph{not} assumed to have characteristic zero.

 By $\Ch$ we denote the symmetric monoidal category of differential graded vector spaces over $k$ (aka chain complexes over $k$) equipped with its projective model structure in which weak equivalences (for short: equivalences) are quasi-isomorphisms and fibrations are degree-wise surjections. 
Whenever we say that two complexes are canonically equivalent,
this will not necessarily mean that there is a canonical map between which is an equivalence, but more generally a zig-zag of such maps. We will denote equivalences and zig-zags thereof by the symbol $\simeq$.

 A (small) category enriched over $\Ch$ will be called a differential graded category. Unless otherwise stated, functors between differential graded categories will automatically be assumed to be enriched. Note that $\Ch$ is a differential graded category itself.
 
 For a category $\cat{C}$, the sets (or in the enriched setting: space, complex, $\dots$)
 of morphisms from $X \in \cat{C}$ to $Y \in \cat{C}$ will be denoted by $\cat{C}(X,Y)$.

\subparagraph{Acknowledgements.} 
We would like to thank 
	Adrien Brochier, 
	Jürgen Fuchs,
	David Jordan,
	André Henriques,
	Simon Lentner,
	Ehud Meir,
	Svea Nora Mierach,
	Lukas Müller 
	and Yorck Sommerhäuser
	for helpful discussions.

	CS and LW are supported by the RTG 1670 ``Mathematics inspired by String theory and Quantum
	Field Theory''.

\section{Homotopy coends and Hochschild homology\label{secderivedenrichedcoends}}
In this preliminary section, we discuss a suitable version of a homotopy coend that we obtain by slightly modifying the derived functor tensor product given in \cite{shulman} and \cite[Chapter~9]{riehl}.
We can see the construction also as the Hochschild-Mitchell chains on a differential graded category \cite{keller,cibils} with coefficients in a bimodule.
For the convenience of the reader, the presentation will be self-contained.

For differential graded categories $\cat{C}$ and $\cat{D}$,
 we denote by  $\cat{C} \otimes \cat{D}$ 
the differential graded category
whose objects are pairs $(X,Y) \in \cat{C} \times \cat{D}$ of objects of $\cat{C}$ and $\cat{D}$, which we will also denote as $X\times Y$, and whose morphism complexes are given by
\begin{align}
(\cat{C} \otimes \cat{D})(X \times Y,X'\times Y') := \cat{C}(X,X') \otimes \cat{D}(Y,Y') \quad \text{for} \quad X,X' \in \cat{C}\ , \quad Y,Y' \in \cat{D} \ .
\end{align}

\begin{definition}\label{defibar}
	Let $\cat{C}$ be a differential graded category. For a functor $F: \cat{C}^\op \otimes \cat{C} \to \Ch$ (by the above conventions, we will always assume that $F$ is enriched),
	 we define the \emph{(enriched) simplicial bar construction} as the simplicial object in $\Ch$ with $n$-simplices
	\begin{align} B_n F := \bigoplus_{X_0,\dots,X_n \in \cat{C}} \cat{C}(X_1,X_0) \otimes \dots \otimes \cat{C}(X_{n},X_{n-1}) \otimes F(X_0,X_n)\ .\end{align} 
	The face and degeneracy maps are defined similarly to \cite[Definition~9.1.1]{riehl}. The $j$-th face map composes morphisms over the $j$-th object thereby deleting it from the list of objects indexing the summand
	and the $j$-degeneracy maps inserts an identity at the $j$-th object thereby doubling it in the list of objects indexing the summand. In more detail, we have:
	\begin{itemize}
		
		\item The face map $\partial_0 : B_n F \to B_{n-1} F$ is induced by the map
		\begin{align}
		\cat{C}(X_1,X_0) \otimes F(X_0,X_n) \to F(X_1,X_n)
		\end{align} which is part of the data of $F$ being an enriched functor. 
		
		\item For $0<j<n$, the face map $\partial_j : B_n F \to B_{n-1} F$ arises from the composition map
		\begin{align}
		\cat{C}(X_{j},X_{j-1})\otimes \cat{C}(X_{j+1},X_{j}) \to \cat{C}(X_{j+1},X_{j-1}) \ .
		\end{align}

		\item The face map $\partial_n : B_n F \to B_{n-1} F$ is induced by the map
		\begin{align}
		\cat{C}(X_{n},X_{n-1}) \otimes F(X_0,X_n) \to F(X_{0},X_{n-1})
		\end{align} which is part of the data of $F$ being an enriched functor. 
		
		\item The degeneracy map $s_j : B_n F \to B_{n+1} F$ inserts an identity at $X_j$ using the canonical map $k \to \cat{C}(X_j,X_j)$ selecting the identity. 
		
	\end{itemize}
	
	We define the \emph{homotopy  coend} of $F$ as the realization of $B_* F$, i.e.\ by
	\begin{align} \int_\mathbb{L}^{X \in \cat{C}} F(X,X) := |B_* F|= \int^{n \in \Delta^\op} N_*(\Delta^n;k) \otimes B_n F\ , \end{align} where $N_*(\Delta^n;k)$ are the normalized chains on the standard simplex $\Delta^n$ with coefficients in $k$
	(equivalently, we may see $B_*F$ as a double complex and totalize).
\end{definition}


For a differential graded category $\cat{C}$,
 a functor $F:\cat{C}^\op \otimes \cat{C} \to \Ch$ will also be referred to as \emph{$\cat{C}$-bimodule}. The homotopy coend $\lint^{X \in \cat{C}} F(X,X)$ will also be called the \emph{derived trace of $F$}.

The above constructions can be done for more general model categories than chain complexes over a field. But since we want to develop the techniques with an eye towards the intended applications,
we deliberately reduce the generality.  However, for a slight generalization that we need later on see Remark~\ref{remgendercoend}.

\begin{example}\label{exhochschildalgebra}
	For any $k$-algebra $A$ (by this we always mean an associative and unital $k$-algebra), the category $\star // A$ with one object whose endomorphisms are given by $A$ is a differential graded category (in fact, a linear category in this case), and a functor $\left( \star //A \right)^\op \otimes (\star // A) \to \Ch$ is a differential graded module $M$ over the enveloping algebra $A\e = A^\op \otimes A$ of $A$, i.e.\
	an $A$-bimodule. Now the homotopy coend of $M$ over $\star //A$ is just given by the Hochschild chains for the algebra $A$ with coefficients in the $A\e$-module $M$ which we denote by $CH(A;M)$. 
\end{example}

\begin{remark}\label{remarkreversal}
	For a differential graded category $\cat{C}$ and a $\cat{C}$-bimodule $F:\cat{C}^\op \otimes \cat{C} \to \Ch$, we obtain a $\cat{C}^\op$-bimodule $F^\op:\cat{C}\otimes\cat{C}^\op \to \Ch$ by precomposition of $F$ with the flip map. 
	Then by reading backwards the families of objects used for the definition of the bar construction of $F$, we obtain a reversal isomorphism
	$\lint^{X \in \cat{C}^\op} F^\op(X,X)\cong\lint^{X\in\cat{C}}F(X,X)$.
\end{remark}

One key feature of the homotopy  coend 
is its homotopy invariance:

\begin{proposition}\label{homotopyinvariancepropo}
	Let $\cat{C}$ be a differential graded category. Then an equivalence $F \xrightarrow{\ \simeq\ } G$  between functors $F,G: \cat{C} ^\op \otimes \cat{C} \to \Ch$ induces an equivalence
	\begin{align}
	\lint^{X \in \cat{C}} F(X,X)\xrightarrow{\ \simeq\ } \lint^{X \in \cat{C}} G(X,X)\ .
	\end{align}
\end{proposition}


The homotopy invariance for the homotopy coend is an extremely crucial  built-in property (otherwise the homotopy coend would not deserve its name); it will be used without further mention.
The proof of Proposition~\ref{homotopyinvariancepropo} readily follows from the following Lemma, which we obtain from the fact that
every simplicial vector space is Reedy cofibrant and \cite[Theorem~19.8.4 (1)]{Hirschhorn} applied
to the framing obtained by normalized chains on the standard simplices:

\begin{lemma}\label{lemmahocolim}
	Let $\cat{C}$ be a differential graded category and $F:\cat{C}^\op \otimes \cat{C} \to \Ch$ a functor.
	Then there is a natural equivalence
	\begin{align}
	\lint^{X\in\cat{C}} F(X,X) \simeq \hocolimsub{n \in \Delta^\op}\,  B_n F \ . 
	\end{align}
\end{lemma}



\subsection{Yoneda Lemma and Fubini Theorem}
Next we discuss some important tools which will help us to compute with homotopy  coends. They generalize to some extent 
the calculus for ordinary coends \cite{maclane}. We will often encounter the requirement that both the differential graded category and the bimodule are concentrated in non-negative degree (or more generally, bounded below).
First, we formulate the Yoneda Lemma:

\begin{proposition}
	\label{propyoneda}
	Let $\cat{C}$ be a differential graded category and $H : \cat{C} \to \Ch$ a functor 
	such that
	both the morphism spaces of $\cat{C}$ and the values of $H$ are concentrated in non-negative degree.
	Then there is a canonical equivalence
	\begin{align}
	\lint^{X \in \cat{C}} \cat{C}(X,-) \otimes H(X) \xrightarrow{\ \simeq\ } H\ . \label{eqnyonedaeqn}
	\end{align}
	Moreover, this map is surjective and hence a trivial fibration.
\end{proposition}

\begin{proof}
	For $Y \in \cat{C}$, the maps $\cat{C}(X,Y) \otimes H(X) \to H(Y)$ provide an augmentation $B_* (\cat{C}(-,Y) \otimes H) \to H(Y)$ whose realization is a map
	\begin{align}
	\lint^{X \in \cat{C}} \cat{C}(X,Y) \otimes H(X) \to H(Y)\label{eqnyonedaeqn2}
	\end{align}  which is natural in $Y$ and therefore gives us the map \eqref{eqnyonedaeqn}. 
	This map is clearly surjective.

	It remains to show that for fixed $Y$ the map \eqref{eqnyonedaeqn2} is an equivalence: 
	Thanks to the assumptions on $\cat{C}$ and $H$, the homotopy coend 
	$	\lint^{X \in \cat{C}} \cat{C}(X,-) \otimes H(Y)$
	is the realization of the simplicial object $B_* (\cat{C}(-,Y) \otimes H)$ in the simplicial model category of non-negatively graded chain complexes over $k$.
	By \cite[Corollary~4.5.2]{riehl} the augmentation \eqref{eqnyonedaeqn2} is an equivalence if we can exhibit extra degeneracies for the augmentation, see \cite[Section~III.5]{goerssjardine} for a definition of this notion. 
	
	We construct extra degeneracies for the augmentation as follows: 
	For $n\ge 0$, the summand of $B_n (\cat{C}(-,Y) \otimes H)$ belonging to a family $(X_0,\dots,X_n)$ of objects in $\cat{C}$ is given by
	\begin{align}
	\cat{C}(X_1,X_0) \otimes \dots \otimes \cat{C}(X_n,X_{n-1}) \otimes \cat{C}(X_0,Y) \otimes H(X_n) \ . 
	\end{align} 
	Using the identity of $Y$ and the symmetric braiding on $\Ch$, this summand admits a natural map to
	\begin{align}
	\cat{C}(X_0,Y) \otimes 	\cat{C}(X_1,X_0) \otimes \dots \otimes \cat{C}(X_n,X_{n-1})   \otimes \cat{C}(Y,Y) \otimes H(X_n) \ , 
	\end{align}
	i.e.\ to the summand of
	$B_{n+1} (\cat{C}(-,Y) \otimes H)$ 
	belonging to $(Y,X_0,\dots,X_n)$.
	This yields a map $s_{-1}^n : B_n (\cat{C}(-,Y) \otimes H) \to B_{n+1} (\cat{C}(-,Y) \otimes H)$. It is straightforward to verify that this gives us extra degeneracies for the augmentation.
\end{proof}

For  homotopy coends, there is a Fubini Theorem making a statement about the `order of integration' for iterated coends:

\begin{proposition}\label{fubinisthmprop}
	Let $\cat{C}$ and $\cat{D}$ be differential graded categories and $F : \left(   \cat{C}  \otimes \cat{D} \right)  ^\op \otimes \cat{C} \otimes \cat{D}\to \Ch$  
	a functor.
	Then there is a  natural isomorphism\label{fubinisthmpropa}
	\begin{align}
	\lint^{X \in \cat{C}} \lint^{Y \in \cat{D}} F(  X \times Y,X\times Y     ) \cong \lint^{Y \in \cat{D}} 	\lint^{X \in \cat{C}}  F(  X \times Y,X\times Y     ) \ .
	\end{align}
\end{proposition}

\begin{proof}
	From the definitions we  obtain\small
	\begin{align}
	&	\lint^{X \in \cat{C}} \lint^{Y \in \cat{D}} F(  X \times Y,X\times Y     ) \\= &\int^{m\in \Delta^\op} N_*(\Delta^m;k) \otimes \left(  \bigoplus_{X_0,\dots,X_m \in \cat{C}} \cat{C}(X_1,X_0) \otimes \dots \otimes \cat{C}(X_{m},X_{m-1})    \right. \\ &\left.  \otimes \int^{n\in \Delta^\op} N_*(\Delta^n;k)  \otimes \left(    \bigoplus_{Y_0,\dots,Y_n \in \cat{D}} \cat{D}(Y_1,Y_0) \otimes \dots \otimes \cat{D}(Y_{n},Y_{n-1}) \otimes F(X_0\times Y_0,X_m\times Y_n)           \right)    \right)
	\ .
	\end{align}\normalsize
	Using that tensor products of chain complexes,
	direct sums and coends commute, we see that this is canonically 
	isomorphic to $\lint^{Y \in \cat{D}} 	\lint^{X \in \cat{C}}  F(  X \times Y,X\times Y     )$. 
\end{proof}

The above Proposition tells us that the `order of integration does not matter'.
Therefore, instead of $	\lint^{X \in \cat{C}} \lint^{Y \in \cat{D}}$ or $\lint^{Y \in \cat{D}}	\lint^{X \in \cat{C}} $, we will just write $\lint^{\substack{ X\in \cat{C} \\ Y \in \cat{D}}}$.

\subsection{Agreement principle\label{secderivedcoendsproj}}
In the sequel, we will often have to evaluate homotopy coends over (sub)categories of modules over some algebra. It is a pertinent question whether such a homotopy coend can be reduced to a homotopy coend over the one-object category associated to that algebra and hence to (ordinary) Hochschild chains (Example~\ref{exhochschildalgebra}). 
This leads us to an Agreement Principle that goes back to \cite{mcarthy,keller},
where it appears in a slightly different form. We explain the relation after Corollary~\ref{corprojectiveforbimodules}.

We refer to a $k$-linear category as a \emph{finite-dimensional algebroid over $k$} if it is equivalent to a $k$-linear category with finitely many objects and finite-dimensional morphism spaces. The main example is the one-object category whose endomorphisms are given by a finite-dimensional $k$-algebra.

For  a finite-dimensional algebroid $\cat{A}$ over $k$,
we denote by 
$\Mod_k \cat{A}$ the $k$-linear category of all finite-dimensional $\cat{A}$-modules.
Here, a finite-dimensional $\cat{A}$-module is a $k$-linear functor from $\cat{A}$ to finite-dimension\-al $k$-vector spaces.
By $\Proj_k \cat{A}\subset \Mod_k \cat{A}$ we denote the full $k$-linear subcategory  of finite-dimension\-al projective $\cat{A}$-modules. 
We refer to
\cite[Section~2.2]{weibel}
for the usual equivalent descriptions of projective modules.
Following our general conventions for the notation, we denote by $\cat{A}(-,-)$, $\Mod_k \cat{A}(-,-)$ and $\Proj_k \cat{A}(-,-)$ the morphism spaces in these $k$-linear categories.

There is a canonical embedding
$
\iota_\cat{A} : \cat{A}^\op \to \Proj_k \cat{A}
$ sending $a \in \cat{A}$ to $\cat{A}(a,-)$ along which we can restrict homotopy coends over $\Proj_k \cat{A}$. To make a statement about such restricted homotopy coends, we need the following Lemma:

\begin{lemma}\label{lemmadualhom}
	For any 
	finite-dimensional algebroid $\cat{A}$ over $k$,
	any functor $F:\left( \Proj_k \cat{A}\right)^\op \otimes \Proj_k \cat{A} \to \Ch$
	whose values are concentrated in non-negative degree
	and 
	$X,Y\in \Proj_k \cat{A}$ 
	the natural map
	\begin{align}	\lint^{a\in \cat{A}}           \Proj_k \cat{A}(X,\iota_\cat{A}(a)) \otimes F(\iota_\cat{A}(a), Y) \xrightarrow{\  \simeq  \    }  F(X,Y)    \label{eqnmapagreementlemma}
	\end{align} is an equivalence. 
\end{lemma}

\begin{proof} 
	Without loss of generality, we can assume that $\cat{A}$ has finitely many objects and finite-dimensional morphism spaces.
	We now observe that for a fixed projective $\cat{A}$-module $Y$, the statement that the map
	\eqref{eqnmapagreementlemma} is an equivalence
	is true in the following cases:
	\begin{xenumerate}
		\item  For $X=\cat{A}(b,-)$ for any $b\in \cat{A}$, it is true by the Yoneda Lemma (Proposition~\ref{propyoneda}), \label{statementfree}
		\item  It is true for finite-dimensional $\cat{A}$-modules $X$ and $X'$ if and only if it is true for $X\oplus X'$. \label{statementsum}
		Here we use that $F$ by our conventions is always assumed to be enriched. As a consequence, it preserves finite biproducts, i.e.\ $F(X\oplus X',Y)\cong F(X,Y)\oplus F(X',Y)$.
	\end{xenumerate}
	Now if $X$ is an \emph{arbitrary} finite-dimensional and projective $\cat{A}$-module, then the finite-dimensional $\cat{A}$-module $Z:= \bigoplus_{a \in \cat{A}} \cat{A}(a,-) \otimes X(a)$ comes with a canonical surjection $\pi : Z \to X$,
	and the short exact sequence \begin{align} 0 \to \ker \pi \to Z \to X\to 0\end{align} splits by projectivity of $X$, and hence $Z \cong X \oplus \ker \pi$.
	From \ref{statementfree} and \ref{statementsum} it follows that the statement is true for $Z$ and therefore also for $X$ by \ref{statementsum}.
\end{proof}


\begin{theorem}[Agreement principle]\label{theoprojectiveforbimodules}
	Let $\cat{A}$ be a finite-dimensional algebroid over $k$
	and $F: \left( \Proj_k \cat{A}\right)^\op \otimes \Proj_k \cat{A} \to \Ch$ a functor whose values are concentrated in non-negative degree. 
	Then the canonical embedding $\iota_\cat{A} : \cat{A}^\op \to \Proj_k \cat{A}$ induces an equivalence
	\begin{align}
	\lint^{a\in \cat{A}}      F(   \iota_\cat{A}(a) , \iota_\cat{A}(a)    ) \xrightarrow{\  \simeq  \    } \lint^{X\in \Proj_k \cat{A}} F(X,X) \ . \label{mapthmprojective}
	\end{align}
\end{theorem}

\begin{proof}
	The map \eqref{mapthmprojective}
	is the composition of the reversal isomorphism \begin{align}
	\lint^{a\in \cat{A}}      F(   \iota_\cat{A}(a) , \iota_\cat{A}(a)    )\cong \lint^{a\in \cat{A}^\op}      F^\op(   \iota_\cat{A}(a) , \iota_\cat{A}(a)    )\end{align}
	from Remark~\ref{remarkreversal}
	with the map 
	\begin{align}
	\iota_\cat{A}:	\lint^{a\in \cat{A}^\op}      F^\op(   \iota_\cat{A}(a) , \iota_\cat{A}(a)    )\to \lint^{X\in \Proj_k \cat{A}} F(X,X)\label{mapthmprojectiveaux}
	\end{align} induced directly by the embedding $\iota_\cat{A}$. 
	Hence, it suffices to prove that \eqref{mapthmprojectiveaux} is an equivalence. To this end,
	we note that it fits into the square\footnotesize
	\begin{equation*}\label{commutingsquaretransf}
	\begin{tikzcd}
	\displaystyle\lint^{a\in \cat{A}^\op}      F^\op(   \iota_\cat{A}(a) , \iota_\cat{A}(a)    )  \ar{rr}{\iota_\cat{A}}&& \displaystyle\lint^{X \in \Proj_k \cat{A}} F(X,X)   \\
	\displaystyle\lint^{a\in \cat{A}^\op}   \lint^{X \in \Proj_k \cat{A}} \Proj_k \cat{A}(X,\iota_\cat{A}(a)) \otimes F^\op(\iota_\cat{A}(a),X)   \ar{rr}{\varphi }   \ar{u}{\alpha }  && \displaystyle\lint^{X \in \Proj_k \cat{A}   } \lint^{a\in \cat{A}^\op}           \Proj_k \cat{A}(X,\iota_\cat{A}(a)) \otimes F^\op(\iota_\cat{A}(a),X)\ ,     \ar{u}{\beta }
	\end{tikzcd} 
	\end{equation*}	
	\normalsize
	where $\alpha$ is the natural equivalence from the Yoneda Lemma (Proposition~\ref{propyoneda}), $\beta$ is the equivalence from Lemma~\ref{lemmadualhom} (combined with Remark~\ref{remarkreversal}),
	and the isomorphism $\varphi$ is a consequence of the Fubini Theorem (Proposition~\ref{fubinisthmprop}).
	It remains to prove that the square commutes up to homotopy because then we may conclude that \eqref{mapthmprojectiveaux} is an equivalence.
	
	To prove this, we first note that we can see all the 
	complexes in the above square as realizations of simplicial complexes, namely the simplicial bar constructions that we have used to define homotopy coends (the two lower complexes are iterated homotopy coends, hence they are even bisimplicial); moreover, all the maps involved arise as simplicial maps between the simplicial bar constructions. 
Therefore, we may as well exhibit a simplicial homotopy
	$\iota_{\cat{A}} \alpha \simeq \beta \varphi$, see \cite[Definition~8.3.11]{weibel} for the definition. 
	The complex in the left lower corner of the square can be modeled by the total complex associated to a bisimplicial object, 
	and the latter
	 can be described as the realization of the diagonal simplicial object
	 by the generalized Eilenberg-Zilber Theorem of Dold and Puppe
	 \cite[IV.2~Theorem~2.4]{goerssjardine}.
	 Therefore, the needed simplicial homotopy will run from
	 the simplicial chain complex which in degree $n$ is given by a direct sum of 
	 (the reader may ignore the $(*)$-labeled underbraces for the moment)
	\begin{align}\begin{array}{c}
	\cat{A}(a_0,a_1) \otimes \dots \otimes \cat{A}(a_{j-1},a_j) \otimes \underbrace{\cat{A}(a_j,a_{j+1}) \otimes \dots \otimes  \cat{A}(a_{n-1},a_n)}_{(*)}  \\  \otimes \underbrace{\Proj_k \cat{A}(X_1,X_0) \otimes \dots \otimes \Proj_k \cat{A}(X_{j},X_{j-1})}_{(*)} \otimes \Proj_k \cat{A}(  X_{j+1},X_j    ) \otimes \dots \otimes  \Proj_k \cat{A}(X_n,X_{n-1}) \\ \otimes \underbrace{\Proj_k\cat{A} (X_0,\iota_\cat{A}(a_n))}_{(*)} \otimes F(\iota_\cat{A}(a_0) , X_n)\end{array} \label{eqnsummandsimphomotopy}
	\end{align}
	for $a_0,\dots,a_n \in \cat{A}$ and $X_0,\dots,X_n \in \Proj_k \cat{A}$ 
	to $B_* F$ as given in Definition~\ref{defibar}. 
	For every
	$0\le j\le n$, 
	the tensor factors marked by $(*)$
	admit a map to $\Proj_k \cat{A}(X_j,\iota_\cat{A}(a_j))$
	which uses 
	composition in $\cat{A}$ and $\Proj _k \cat{A}$.
	Combining this with the functor $\iota_\cat{A}$, we obtain a map
 $h_j$ from the summand \eqref{eqnsummandsimphomotopy} to the summand 
		\begin{align}
	\Proj_k \cat{A}(\iota_\cat{A}(a_1),\iota_\cat{A}(a_0)) \otimes \dots \otimes \Proj_k \cat{A}(\iota_\cat{A}(a_j),\iota_\cat{A}(a_{j-1})) \otimes \Proj_k \cat{A}(X_j,\iota_\cat{A}(a_j))\\  \otimes \Proj_k \cat{A}(X_{j+1},X_j) \otimes \dots \otimes \Proj_k \cat{A}(X_{n},X_{n-1})\otimes 
	F(\iota_\cat{A}(a_0),X_{n}) \ . 
	\end{align}
	of $B_{n+1} F$ indexed by
	$\iota_\cat{A}(a_0),\dots,\iota_\cat{A} (a_j),X_j,\dots,X_n$. 
	As can be verified by a direct computation,
	these maps yield a simplicial homotopy
	from $\partial_0 h_0 = \beta \varphi$ to $\partial_{n+1} h_n = \iota_\cat{A} \alpha$. 
\end{proof}

As an important application, 
Theorem~\ref{theoprojectiveforbimodules}
yields:

\begin{corollary}\label{corprojectiveforbimodules}
	For any 
	finite-dimensional algebroid $\cat{A}$ over $k$, the canonical map
	\begin{align}
	\lint^{a\in \cat{A}} \cat{A}(a,a) \xrightarrow{\ \simeq \ } \lint^{X\in\Proj_k \cat{A}}           \Proj_k \cat{A}(X,X)
	\end{align} is an equivalence. 
\end{corollary}

\begin{example}\label{exprojectiveforbimodules}
	If $\cat{A}=\star // A$ for a finite-dimensional $k$-algebra $A$ and if $F: \left( \Proj_k \cat{A}\right)^\op \otimes \Proj_k \cat{A} \to \Ch$ satisfies the hypotheses of Theorem~\ref{theoprojectiveforbimodules}, we have
	$
	CH(A;F(A,A)) \simeq \lint^{X\in \Proj_k A} F(X,X) 
	$, 
	where $F(A,A)$ is the $A$-bimodule that we obtain by evaluation of $F$ on the free $A$-module $A$,
	 and $CH(A;F(A,A))$ are the Hochschild chains of $A$ with coefficients in that bimodule (Example~\ref{exhochschildalgebra}). 
	In particular, we obtain the Agreement Principle from
\cite{mcarthy} and	\cite[Theorem~1.5~(a)]{keller}
	that the Hochschild homology of $A$ and the Hochschild
	 homology of the $k$-linear category of finite-dimensional projective $A$-modules are isomorphic. 
\end{example}

\begin{remark}[Generalization]\label{remgendercoend}
	Above, we have defined and investigated homotopy coends of functors going from $\cat{C}^\op \otimes \cat{C}$ for some differential graded category $\cat{C}$ to chain complexes $\Ch$ over a field $k$. In fact, we could have also used chain complexes of modules over an algebra over $k$ instead
	of $\Ch$.
	Let us sketch this generalization: For a functor $F: \cat{C}^\op \otimes \cat{C} \to \ChA$ to chain complexes of modules over some $k$-algebra $R$, we can consider the functor $\cof F : \cat{C}^\op \otimes \cat{C} \to \ChA$ obtained by replacing $F$ cofibrantly pointwise (here we fix again the projective model structure on $\ChA$). Using the tensoring of $\ChA$ over $\Ch$,
	we now define the bar construction $B_* \cof F$ by precisely the same formulae as in Definition~\ref{defibar}. Its realization
	\begin{align}
	\lint^{X \in \cat{C}} F(X,X) := | B_* \cof F|
	\end{align}
	will be referred to as the \emph{homotopy coend of $F$}.
	Having replaced $F$ pointwise will ensure that  $B_* \cof F$ is cofibrant in each level, which implies that $B_* \cof F$ is Reedy cofibrant. 
	As a consequence, the proof of homotopy invariance goes through. This allows us to prove  the generalization of the Yoneda Lemma, the Fubini Theorem and the Agreement Principle.
	
\end{remark}


\section{The Hochschild complex of a finite tensor category}
In this section, we investigate the Hochschild complex of a finite tensor category. 
By means of a resolution of the canonical coend of a finite tensor category,
 we express the Hochschild complex in terms of derived class functions (Section~\ref{sectracesclass}). This will turn out to be the key to a topological interpretation of the Hochschild complex in Section~\ref{secmcg}, where we construct a homotopy coherent action of the torus mapping class group.
 In Sections~\ref{secdgva} and~\ref{secequiv}, we discuss the multiplicative structure on the Hochschild complex of a  braided finite  tensor category.\\

We start by giving definitions and recalling standard terminology:
Based on the comparison between homotopy coends and ordinary Hochschild homology in Section~\ref{secderivedcoendsproj},  the following Definition makes sense: 

\begin{definition}\label{defhochschildcomplex}
	For a $k$-linear category $\cat{C}$,
	 we call the differential graded vector space
	\begin{align}
	\lint^{X\in \Proj \cat{C}} \cat{C}(X,X) \ 
	\end{align} the \emph{Hochschild complex of $\cat{C}$}.
	\end{definition}

	This definition reduces to the standard definition of Hochschild homology of a differential graded (here: just linear) category as appearing in \cite{keller,cibils}. Sometimes it is also referred to as a \emph{derived trace}. Let us emphasize, however, that in the above definition the homotopy coend only runs over the \emph{projective objects}.
	
	In \cite{shimizu} a version of Hochschild cohomology of a finite Abelian linear category is proposed using the category of right exact endofunctors. Hochschild homology is then defined indirectly using the Nakayama functor and a dualization. Due to the strong finiteness conditions in \cite{shimizu},
	 all these definitions are equivalent on their common domain of definition.

Consider 
Definition~\ref{defhochschildcomplex} 
in the case that 
$\cat{C}$ is a
\emph{finite category}, i.e.\
a linear Abelian category 
with finite-dimensional morphism spaces, enough projectives, and finitely many isomorphism classes of simple objects subject to the condition that every object has finite length. A linear category is finite if and only if it is linearly equivalent to the category of finite-dimensional modules over a finite-dimensional algebra, see 
e.g.\  \cite[Proposition~1.4]{dss}. 
If $\cat{C}$ is finite, i.e.\
given by finite-dimensional modules over some finite-dimensional algebra $A$, then $\lint^{X\in \Proj \cat{C}} \cat{C}(X,X)$ is just equivalent to the ordinary Hochschild chains on $A$ with coefficients in the $A$-bimodule $A$.

Since the categories we are interested in will all be of that type, one might ask why it is necessary to consider the complex $\lint^{X\in \Proj \cat{C}} \cat{C}(X,X)$ when it is just equivalent to the Hochschild complex of some algebra. The answer is that just knowing that $\lint^{X\in \Proj \cat{C}} \cat{C}(X,X)$ is equivalent to the Hochschild complex of \emph{some} algebra is often not very helpful because this presentation in terms of an algebra might be non-canonical; in some sense it corresponds to a choice of coordinates.
 As a consequence, constructions performed on $\lint^{X\in \Proj \cat{C}} \cat{C}(X,X)$ might not have a direct counterpart for the Hochschild complex of the randomly chosen algebra (in other cases they might have, but those constructions might be a lot more complicated). This is especially problematic when the category has more structure (tensor product, braiding, ribbon twist). This additional structure might not be reflected on the algebra.
Since the main goal of this article is to investigate the additional structure which is present on $	\lint^{X\in \Proj \cat{C}} \cat{C}(X,X)$ when $\cat{C}$ is a monoidal category or braided monoidal category (possibly with more structure or properties), using Definition~\ref{defhochschildcomplex} is justified.

Before proceeding let us recall some standard notions from the theory of linear monoidal categories:
A \emph{$k$-linear monoidal category} is a monoidal category with $k$-linear monoidal product.
In a \emph{rigid} $k$-linear monoidal category every object $X \in \cat{C}$ has a left dual $X^\vee$ and a right dual ${^\vee \! X}$. These give us the natural adjunction isomorphisms
\begin{align}
\cat{C}(X\otimes Y,Z)&\cong \cat{C}(X,Z\otimes Y^\vee) \ , \\
\cat{C}(Y^\vee\otimes X,Z)&\cong \cat{C}(X,Y\otimes Z) \ , \\
\cat{C}(X\otimes {^\vee \! Y}, Z) &\cong \cat{C}(X,Z\otimes Y)\ ,\\
\cat{C}(Y\otimes X,Z)&\cong \cat{C}(X,{^\vee \! Y} \otimes Z)  
\end{align}
for $X,Y,Z\in\cat{C}$ (we are following here the conventions of \cite{egno}). 
A $k$-linear Abelian rigid monoidal category with simple unit will be referred to as a \emph{tensor category}.  
A \emph{finite tensor category} \cite{etinghofostrik} is a tensor category which is also finite as a linear category.
Such a category has the important property that $P\otimes X$ and $X\otimes P$ are projective for $P \in \Proj \cat{C}$ and $X\in\cat{C}$ and that the tensor product is exact in both arguments. Furthermore, it is self-injective, i.e.\ the projective objects are precisely the injective ones.

\subsection{Hochschild complex of Drinfeld doubles in finite characteristic}
Before investigating the properties of the homotopy coend $	\lint^{X\in \Proj \cat{C}} \cat{C}(X,X)$  in general, it is certainly instructive to look at a certain class of finite tensor categories 
which allow us to perform some concrete computations, namely \emph{Drinfeld doubles}.
Recall from e.g.\ \cite[Chapter~IX]{kassel} that for a finite group $G$ the Drinfeld double $D(G)$
is a ribbon factorizable Hopf algebra whose underlying vector space is $k(G) \otimes k[G]$. Here we denote by $k(G)$ the commutative algebra of $k$-valued functions on $G$; a basis will be given by the functions $(\delta_g)_{g\in G}$ supported in a single group element. Moreover, we denote by $k[G]$ the group algebra.
Now the multiplication of $D(G)$ is given by
\begin{align}
(\delta_a \otimes b)(\delta_c \otimes d) = \delta_a \delta_{bcb^{-1}} \otimes bd \quad \text{for all}\quad a,b,c,d\in G \ . 
\end{align}
Modules over $D(G)$ can be equivalently written as Yetter-Drinfeld modules over $k[G]$, see \cite[Theorem~IX.5.2]{kassel}, and hence as modules over the action groupoid $G//G$ of $G$ acting on itself by conjugation;
\begin{align}
\Mod_k D(G) \simeq \Mod_k (G//G) \ . \label{yetterdrinfeldmodeqn}
\end{align}
The groupoid $G//G$ is equivalent to the groupoid $\PBun_G(\mathbb{S}^1)$ of principal $G$-bundles over the circle, and in fact this observation 
is the basis for the description of Dijkgraaf-Witten theory \cite{dijkgraafwitten,freedquinn} in terms of principal $G$-bundles.
We will use \eqref{yetterdrinfeldmodeqn} to give a topological interpretation to $\lint^{X \in \Proj_k D(G)} \Hom_{D(G)}(X,X)$.

To this end, we introduce some notation: For a groupoid $\Gamma$, 
denote by
$\Lambda \Gamma$
	its \emph{loop groupoid}, i.e.\ the groupoid $\Gamma^{\Pi(\mathbb{S}^1)}$ of functors from the fundamental groupoid $\Pi (\mathbb{S}^1)$ of the circle $\mathbb{S}^1$ to $\Gamma$.
	For a group $G$ and the groupoid $BG$ with one object
	and automorphism group $G$, 
	the loop groupoid $\Lambda BG$ is equivalent to the groupoid of principal $G$-bundles over the circle,
	\begin{align}
	\Lambda BG = BG^{\Pi(\mathbb{S}^1)} \simeq \PBun_G(\mathbb{S}^1) 
	\end{align} by the holonomy classification of principal $G$-bundles, 
	and hence equivalent to the action groupoid $G//G$. 
	A similar computation shows
	$
	\Lambda^n BG \simeq \PBun_G(\mathbb{T}^n)
$ and in particular
	\begin{align}
	\Lambda (G//G) \simeq \PBun_G(\mathbb{T}^2)\ . 
	\label{loopgroupoidgadj}
	\end{align}

The chains on the loop groupoid $\Lambda \Gamma$ of any groupoid $\Gamma$ are equivalent to the Hochschild chains of the free $k$-linear category $k[\Gamma]$ on $\Gamma$. This can be seen as a groupoid version of the classical result \cite[Corollary~9.7.5]{weibel}.

\begin{lemma}\label{lemmaloopgroupoidHH}
	For any groupoid $\Gamma$,
	 there is an equivalence
	\begin{align}
	\lint^{x\in \Gamma } k[ \Gamma   ](x,x)    \simeq  N_*  (  \Lambda \Gamma ;k         )   \ . 
	\end{align}
\end{lemma}

\begin{proof}
	Up to equivalence, we can describe $\Pi(\mathbb{S}^1)$ as the groupoid $\star // \mathbb{Z}$ with one object and automorphism group $\mathbb{Z}$.
	As an abbreviation, we will write $S_n$ for the space of $n$-simplices of the simplicial bar construction of $k[\Gamma](-,-)$. Taking Remark~\ref{remarkreversal} into account we can write
	\begin{align} S_n = \bigoplus_{x_0,\dots,x_n \in \Gamma} k[\Gamma(x_0,x_1) ] \otimes \dots \otimes k[\Gamma(x_n,x_0)] \ . 
	\end{align}
	A string of $n$ morphisms in $\Lambda \Gamma$ is a commutative diagram
	\begin{equation}
	\begin{tikzcd}
	x_0 \ar{rr}{\varphi_0} \ar{d}{\alpha_0}&& x_1 \ar{d}{\alpha_1} \ar{rr}{\varphi_1}  && \dots \ar{rr}{\varphi_{n-1}} && x_n \ar{d}{\alpha_n} \\
	x_0 \ar{rr}{\varphi_0} && x_1  \ar{rr}{\varphi_1}  && \dots \ar{rr}{\varphi_{n-1}} && x_n \ . 
	\end{tikzcd}\end{equation} 
	Sending this string to the loop
	\begin{align}
	x_0 \xrightarrow{  \ \varphi_0 \  } x_1 \to \dots \to x_n \xrightarrow{  \   (\varphi_{n-1}\dots\varphi_0)^{-1}    \alpha_n\   } x_0 \in S_n \ 
	\end{align}
	yields an isomorphism from the free simplicial vector space $k[B\Lambda \Gamma]$ on the nerve $B\Lambda \Gamma$ of the loop groupoid of $\Gamma$ to $S_*$. By taking normalized chains, the claim follows.
\end{proof}


\begin{proposition}\label{propodrinfeldmcgaction}
	Let $G$ be a finite group. Then there is an equivalence
	of differential graded vector spaces
	\begin{align}	\lint^{X \in \Proj_k D(G)} \Hom_{D(G)}(X,X) \simeq N_*(\PBun_G(\mathbb{T}^2);k)  \ . \label{eqnodrinfeldmcgaction}
	\end{align}
\end{proposition}

\begin{proof}
	The $k$-linear categories of finite-dimensional representations of $D(G)$ and $G//G$ are 
	equivalent by \eqref{yetterdrinfeldmodeqn}, and hence so are their Hochschild chains. If we apply Corollary~\ref{corprojectiveforbimodules}
	to the free $k$-linear category  $k[G//G]$, we arrive at 
	\begin{align}
	\lint^{X \in \Proj_k D(G)}  \Hom_{D(G)} (X,X) \simeq  	\lint^{g\in G//G } k[ G//G   ](g,g) \ . 
	\end{align} Now we use Lemma~\ref{lemmaloopgroupoidHH} and \eqref{loopgroupoidgadj}. 
\end{proof}

The left hand side of \eqref{eqnodrinfeldmcgaction}
is also equivalent to the ordinary Hochschild chains on $D(G)$ with coefficients in $D(G)$ seen as bimodule over itself (Example~\ref{exprojectiveforbimodules}).
By transporting the geometric mapping class group action from $N_*(\PBun_G(\mathbb{T}^2);k)$ to the left hand side we obtain:

\begin{corollary}
For any finite group $G$, the Hochschild chains of the Drinfeld double $D(G)$ carry a homotopy coherent $\SL(2,\mathbb{Z})$-action.
\end{corollary}

The category of modules over a Drinfeld double is a very tractable example of a modular category 
(the definition of a modular category will be recalled at the beginning of Section~\ref{secmcg}).
It is non-semisimple if and only if the characteristic of $k$ divides $|G|$. The above result establishes a homotopy coherent mapping class group action on its Hochschild complex. As one of the main results of this article (Theorem~\ref{thmsl2z}), we generalize this to arbitrary modular categories. In the general case, 
a geometric argument as in Proposition~\ref{propodrinfeldmcgaction} will not be available.

\subsection{Traces, class functions and the Lyubashenko coend\label{sectracesclass}}
If we are given a finite tensor category $\cat{C}$ and consider its Hochschild complex
$\lint^{X\in \Proj \cat{C}} \cat{C}(X,X)$, then the tensor structure is of course not needed to define the complex itself. However, it leads to simplifications when trying to compute the Hochschild homology.
The idea is to express a (derived) trace (i.e.\ a (homotopy) coend of some sort) via a (derived) space of class functions.

Before making this idea precise below in the case of interest to us, we explain in more detail a related instance where it appears in a different form:
Let $\cat{C},\cat{D}$ and $\cat{E}$ be finite categories over $k$ and $F:\cat{D} \otimes \cat{C}^\op \otimes \cat{C}\to\cat{E}$ a  linear functor. Of course, 
if $\cat{E}$ has sufficiently many colimits, we may consider the coend $\int^{X \in \cat{C}} F(-,X,X)$ to obtain a functor $\cat{D} \to \cat{E}$. However, if $F$ is left-exact, we might want to consider a coend of functors such that the result is again left-exact and such that the universality of the coend holds with respect to left-exact functors. Such requirements arise in conformal field theory for the gluing of conformal blocks. Motivated by this problem, a coend 
$\oint^{X \in \cat{C}} F(-,X,X)$ with values 
in left-exact functors was studied in \cite{lubalex}, see also \cite{fscoend} for a review and the relation to conformal field theory.
It is a key insight that the coend in left-exact functors can be represented by a canonical object in the following way:
Let $\cat{C}$ be a finite tensor category. Then one may define the coend
\begin{align}
\mathbb{F}:=\int^{X \in \cat{C}} X \otimes X^\vee
\end{align}
which is called the \emph{canonical coend of $\cat{C}$} or also the \emph{Lyubashenko coend} due to its appearance in \cite{lubacmp,luba,kl}. 
By \cite[Section~8.2]{lubalex} we find
\begin{align}
\oint^{X \in \cat{C}} \cat{C}(X,-\otimes X) \cong \cat{C}(I,-\otimes \mathbb{F}) \ ,
\end{align}
i.e.\ the coend of 
the morphism space functor computed in the category of left-exact functors (which is just a type of trace) can be written as the space of morphisms from the monoidal unit to some special object $\mathbb{F}$. If $\cat{C}$ arises as finite-dimensional modules over a finite-dimensional Hopf algebra, then $\mathbb{F}$ is the coadjoint representation. For this reason, $\cat{C}(I,\mathbb{F})$ should be thought of as a generalized space of class functions. 
In summary, we see an instance where a trace is expressed as a space of class functions.
We should note that the object $\mathbb{F}$ is not only interesting because it provides a description of certain coends in left-exact functors. It also turns out to be the key ingredient for the construction of the mapping class group actions in \cite{lubacmp,luba}, see also Section~\ref{secmcg}.

This suggests the question whether we can describe for a finite tensor category $\cat{C}$ in an analogous way the derived trace  $\lint^{X\in \Proj \cat{C}} \cat{C}(X,X)$, i.e.\ the Hochschild complex, as a derived space of class functions using some special differential graded object of $\cat{C}$; and if so, whether it is related to the Lyubashenko coend. In order to answer these questions, we consider for any functor $F: \cat{C}^\op \otimes \cat{C} \to \cat{C}$  the homotopy coend
$
\lint^{X \in \Proj \cat{C}} F(X,X) 
$
by means of the generalizations given in Remark~\ref{remgendercoend}. 
Strictly speaking, we cannot see this (as we would like) as a differential graded object in $\cat{C}$ because $\cat{C}$ does not have infinite coproducts, but the definition of the homotopy coend involves coproducts over \emph{all} projective objects.
Fortunately, we know that up to equivalence we can write $\lint^{X \in \Proj \cat{C}} F(X,X)$ using some finite collection of projective objects. By the Agreement Principle, even one suitably chosen projective module will suffice. We will denote such a `finite version' of $\lint^{X \in \Proj \cat{C}} F(X,X)$
by
\begin{align}
\flint^{X \in \Proj \cat{C}} F(X,X) \ .  \label{eqngendercoendf}
\end{align}
This is now a differential graded object in $\cat{C}$ concentrated in non-negative degree.
Up to equivalence, this object is independent of \emph{how} we  make the homotopy coend finite, i.e.\ two `finite versions' are related by a canonical zig-zag of equivalences.
The computation of \eqref{eqngendercoendf} simplifies when $F$ sends pairs of projective objects to projective objects because then the pointwise cofibrant replacement for $F$ is not necessary (see Remark~\ref{remgendercoend}). In that case, \eqref{eqngendercoendf} is level-wise projective.	

Using such finite homotopy coends in $\cat{C}$ we are able to express $\lint^{X \in \Proj \cat{C}}  \cat{C}(G(X),X)$ for any endofunctor $G$ of $\cat{C}$ as a morphism space from the monoidal unit to some object:

\begin{theorem}\label{thmobHH}
	Let $\cat{C}$ be a finite tensor category.
	Then for any linear functor $G:\cat{C} \to \cat{C}$ there is a canonical equivalence
	\begin{align}
	\lint^{X \in \Proj \cat{C}}  \cat{C}(G(X),X) \simeq \cat{C}\left(I, \flint ^{X \in \Proj \cat{C}} X \otimes G(X)^\vee  \right)
	\end{align}
\end{theorem}

\begin{proof}
	By duality and the Agreement Principle we find
	\begin{align}
	\lint^{X \in \Proj \cat{C}}  \cat{C}(G(X),X) \cong \lint^{X \in \Proj \cat{C}}  \cat{C}(I,X\otimes G(X)^\vee) \simeq \flint^{X \in \Proj \cat{C}}  \cat{C}(I,X\otimes G(X)^\vee) \ . \label{homraus1eqn}
	\end{align} 
	We want to compare this with 
	$\cat{C}\left(I, \flint ^{X \in \Proj \cat{C}} X \otimes G(X)^\vee  \right)$, where 
	the object $\flint ^{X \in \Proj \cat{C}} X \otimes G(X)^\vee $ is a special case of the construction
	\eqref{eqngendercoendf}. The point-wise cofibrant replacement that would usually be involved in the definition of this homotopy coend may be omitted 
	because $X \otimes G(X)^\vee$ is projective whenever $X$ is.
	Since $\flint^{X \in \Proj \cat{C}}  \cat{C}(I,X\otimes G(X)^\vee)$ is defined using finite direct sums which are preserved by the hom functor, we now find $\flint^{X \in \Proj \cat{C}}  \cat{C}(I,X\otimes G(X)^\vee) \cong \cat{C}\left(I, \flint ^{X \in \Proj \cat{C}} X \otimes G(X)^\vee  \right)$ which combined with \eqref{homraus1eqn} yields the assertion.
\end{proof}

This statement is not very useful unless we can understand the differential graded object $\flint ^{X \in \Proj \cat{C}} X \otimes G(X)^\vee $ in $\cat{C}$. To this end, we note that for any functor $F: \cat{C}^\op \otimes \cat{C} \to \cat{C}$ there is an augmentation
\begin{align}
\flint^{X  \in \Proj \cat{C}} F(X,X) \to \fint^{X \in \Proj \cat{C}} F(X,X) 
\end{align}
as follows from the definition of the ordinary coend as a coequalizer.
Here the `f' on the right hand side indicates that the same reduction to a finite coend has been used.  
This map is surjective, hence a fibration.
In fact, the zeroth homology of $\flint^{X  \in \Proj \cat{C}} F(X,X)$ is  $\fint^{X \in \Proj \cat{C}} F(X,X)$  as follows again by definition of the ordinary coend.
But in general, there is no reason why $\flint^{X  \in \Proj \cat{C}} F(X,X) $ should be a projective resolution of $\fint^{X \in \Proj \cat{C}} F(X,X)$.
However, we prove below that this will be true when $F$ is exact and 
sends pairs of projective objects to projective objects. 
In this case, we will understand $F$ as a functor $F:\cat{C}^\op \boxtimes \cat{C} \to \cat{C}$, where $\boxtimes$ denotes the Deligne product.
First  observe that by \cite[Proposition~5.1.7]{kl} the exactness of $F$ ensures that the canonical map $\fint^{X \in \Proj \cat{C}} F(X,X) \to \int^{X \in \cat{C}} F(X,X)$ is an isomorphism (this is a statement about ordinary coends). Therefore, we will just write $\int^{X \in \cat{C}} F(X,X)$ instead of $\fint^{X \in \Proj \cat{C}} F(X,X)$. 

\begin{proposition}\label{proporesocoend}
	Let $\cat{C}$ be a finite category and $F:\cat{C}^\op \boxtimes \cat{C} \to \cat{C}$ an exact functor that sends pairs of projective objects to projective objects. Then \begin{align} \flint^{X  \in \Proj \cat{C}} F(X,X) \to \int^{X \in\cat{C}} F(X,X)\end{align} is a projective resolution.
	In particular, for any finite tensor category $\cat{C}$ and any exact functor $G:\cat{C} \to \cat{C}$
\begin{align} \flint^{X \in \Proj \cat{C}} X \otimes G(X)^\vee \to  \int^{X \in \cat{C}} X \otimes G(X)^\vee \end{align} is a projective resolution.
\end{proposition}

\begin{proof}
	By what has just been explained above, it remains to prove
	$H_p\left(  \flint^{X  \in \Proj \cat{C}} F(X,X)  \right)=0$ for $p\neq 0$. 
	For the proof of this fact, we write $\cat{C}$ as finite-dimensional modules over a finite-dimensional algebra $A$. By the Agreement Principle (Theorem~\ref{theoprojectiveforbimodules}) $\flint^{X  \in \Proj \cat{C}} F(A,A)$ is equivalent to the Hochschild chains $A \stackrel{\mathbb{L}}{\otimes}_{A\e}   F(A,A)$ for the $\cat{C}$-valued $A$-bimodule $F(A,A)$. 
	In order to compute the corresponding Hochschild homology, we consider the object $A \boxtimes A \in \cat{C}^\op \boxtimes \cat{C}$, where $A$ acts by right multiplication on the first copy and by left multiplication on the second copy. But $A$ can additionally act from the left on the first copy and from the right on the second copy. This makes $A\boxtimes A$ an $A$-bimodule in $\cat{C}^\op \boxtimes \cat{C}$. The Hochschild complex for this bimodule is given by
	\begin{equation}
	\begin{tikzcd}
\dots \ar[r, shift left=6]  \ar[r, shift left=2]
\ar[r, shift right=6]  \ar[r, shift right=2]
& 	A^{\otimes 2} \bullet \left(   A \boxtimes A    \right) 
\ar[l, shift left=4]  \ar[l]
\ar[l, shift right=4]  
\ar[r, shift left=4] \ar[r, shift right=4] \ar[r] &  A \bullet \left(   A \boxtimes A    \right) \ar[r, shift left=2] \ar[r, shift right=2]
\ar[l, shift left=2] \ar[l, shift right=2]
 & A \boxtimes A\ ,  \ar[l] \\
		\end{tikzcd}
	\end{equation}
	where 	$\bullet$ denotes the tensoring of objects in $\cat{C}^\op\boxtimes \cat{C}$ with vector spaces from the left. This is a complex (or simplicial object) in $\cat{C}\boxtimes \cat{C}$. The underlying complex of vector spaces, however, is just the Hochschild complex for the free $A$-bimodule. Therefore, the augmentation
	\begin{align} A \stackrel{\mathbb{L}}{\otimes}_{A\e} \left(   A \boxtimes A \right) \to A \otimes_{A\e} \left(   A \boxtimes A \right)\label{augeqn1}\end{align} is an equivalence.
	Since $F$ is linear, the augmentation map
	\begin{align} A \stackrel{\mathbb{L}}{\otimes}_{A\e} F(A,A)\to A \otimes_{A\e} F(A,A) \label{augeqn2}\end{align} of the
	Hochschild complex $A \stackrel{\mathbb{L}}{\otimes}_{A\e} F(A,A)$ is the image of the equivalence \eqref{augeqn1} under $F$.
	By exactness of $F$ the map
	\eqref{augeqn2} is now also an equivalence, which proves the claim.
\end{proof}

\begin{corollary} \label{corresocoend2}
For any finite tensor category $\cat{C}$,
the object $\flint^{X \in \Proj \cat{C}} X \otimes X^\vee$ is 
a projective resolution of the canonical coend $\mathbb{F}=\int^{X\in\cat{C}} X\otimes X^\vee$ and allows us to write the Hochschild complex of $\cat{C}$ up to equivalence as
\begin{align} \lint^{X\in\Proj \cat{C}} \cat{C}(X,X) \simeq \cat{C}\left( I, \flint^{X \in \Proj \cat{C}} X \otimes X^\vee  \right) \ .  \end{align} 
\end{corollary}

This finally allows us to describe the Hochschild complex as a generalized space of class functions, i.e.\ as a hom from the monoidal unit to the homotopy Lyubashenko coend.
Note however that Corollary~\ref{corresocoend2} does \emph{not} say that $\lint^{X\in\Proj \cat{C}} \cat{C}(X,X)$ is equivalent to $\cat{C}(I,\cof \mathbb{F})$ for an \emph{arbitrary} projective resolution $\cof \mathbb{F}$ of $\mathbb{F}$. Of course, $\cof \mathbb{F} \simeq \flint^{X \in \Proj \cat{C}} X \otimes X^\vee $ by the essential uniqueness of projective resolutions, but a priori $\cat{C}(I,-)$ need not preserve this equivalence. However, in the following Lemma, whose proof is based on the theory of modified traces \cite{mtrace}, 
we find that, under the assumption that $\cat{C}$ is pivotal, this will be true.

\begin{lemma}\label{lemmaprojCY}
Let $\cat{C}$ be a pivotal tensor category with finite-dimensional morphism spaces
and enough projectives.
Then for $X\in \cat{C}$ the functor $\cat{C}(X,-)$ preserves equivalences between non-negatively differential graded objects which are degree-wise projective.
\end{lemma}

\begin{proof}
Let $\alpha$ be the
socle of
the projective cover of the monoidal unit
and consider  the right modified $\alpha$-trace on the tensor ideal of projective objects \cite[Section~5.3]{mtrace}. This trace in particular provides non-degenerate pairings
\begin{align}
\cat{C}(X,P) \otimes \cat{C}(\alpha \otimes P , X) \to k \label{nondegeneratepairingeqn}
\end{align}
for $X \in \cat{C}$ and $P \in \Proj \cat{C}$ which are moreover natural in $X$ and $P$. In particular, $\cat{C}(X,P)\cong \cat{C}(\alpha \otimes P , X)^*$ by natural isomorphisms.

Now let $P \to Q$ be an equivalence of non-negatively differential graded objects in $\cat{C}$ which are degree-wise projective. 
We need to show that 
for $X \in \cat{C}$ the induced map $\cat{C}(X,P) \to \cat{C}(X,Q)$ is an equivalence. Using the non-degenerate pairing \eqref{nondegeneratepairingeqn} we can equivalently show 
that 
the induced map $\cat{C}(P,{^\vee \alpha} \otimes X)^* \to \cat{C}(Q,{^\vee \alpha} \otimes X)^*$
is an equivalence.
 Since this is a map of finite-dimensional differential graded vector spaces, it suffices to show that the dual map $\cat{C}(Q,{^\vee \alpha} \otimes X) \to \cat{C}(P,{^\vee \alpha} \otimes X)$ is an equivalence. But this is a standard fact from homological algebra, see e.g.\
\cite[Theorem~7.5]{iversen}.
\end{proof}

\begin{theorem}\label{theoderivedcoendviaobject}
Let $\cat{C}$ be a pivotal finite tensor category.
Then for any exact functor $G:\cat{C} \to \cat{C}$, there is a canonical equivalence
\begin{align}
\lint^{X \in \Proj \cat{C}}  \cat{C}(G(X),X) \simeq \cat{C}\left(I,  \cof \int^{X \in \cat{C}} X \otimes G(X)^\vee   \right) \ ,     \label{eqnderivedcoendviaobject}
\end{align}
where $\cof \int^{X \in \cat{C}} X \otimes G(X)^\vee$ is an arbitrary projective resolution of  $\int^{X \in \cat{C}} X \otimes G(X)^\vee$. In particular,
\begin{align}
\lint^{X \in \Proj \cat{C}}  \cat{C}(X,X) \simeq \cat{C}\left(I,  \cof \mathbb{F} \right) \ . 
\end{align}
\end{theorem}

\begin{proof} By Theorem~\ref{thmobHH} we know 
\begin{align}
\lint^{X \in \Proj \cat{C}}  \cat{C}(G(X),X) \simeq \cat{C}\left(I,  \flint^{X \in \Proj \cat{C}} X \otimes G(X)^\vee   \right) \ ; 
\end{align} moreover, $\flint^{X \in \Proj \cat{C}} X \otimes G(X)^\vee $ is a projective resolution of $\int^{X \in \cat{C}} X \otimes G(X)^\vee$ by Proposition~\ref{proporesocoend}. 
Now \eqref{eqnderivedcoendviaobject} holds  for any projective resolution of $\int^{X \in \cat{C}} X \otimes G(X)^\vee$ because $\cat{C}(I,-)$ 
preserves equivalences between two projective resolutions by Lemma~\ref{lemmaprojCY}.
\end{proof}

\subsection{A differential graded version of the Verlinde algebra\label{secdgva}}	
This subsection is concerned with the additional structure on the Hochschild complex of a finite tensor category that additionally has a \emph{braiding}.

This is motivated as follows: It was briefly explained in the introduction that for a semisimple modular category $\cat{C}$, we obtain an algebra structure on the vector space $\int^{X \in \cat{C}} \cat{C}(X,X)$. This can be seen most conceptually by constructing the (anomalous) 3-2-1-dimensional topological field theory associated to $\cat{C}$ and by evaluating it on the torus. Then the multiplication on $\int^{X \in \cat{C}} \cat{C}(X,X)$ comes from the evaluation of this topological field theory on the bordism $P\times \mathbb{S}^1:\mathbb{T}^2 \sqcup \mathbb{T}^2 \to \mathbb{T}^2$, where $P: \mathbb{S}^1 \sqcup \mathbb{S}^1 \to \mathbb{S}^1$ is the pair of pants. We depict $P\times \mathbb{S}^1$ suggestively as:
\begin{center}
	\centering
	\includegraphics[width=0.25\textwidth]{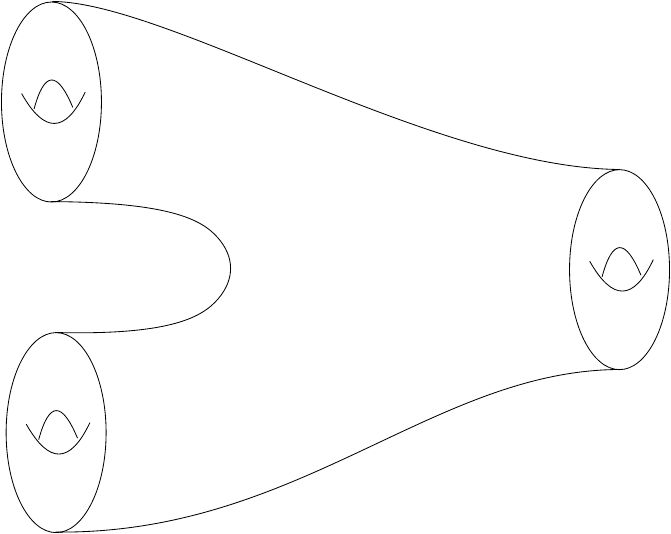}
\end{center}
This multiplication is easily seen to be commutative thanks to the braiding on $\cat{C}$. The resulting algebra is sometimes called the \emph{Verlinde algebra} of $\cat{C}$.

When considering the complex $\lint^{X \in \Proj \cat{C}} \cat{C}(X,X)$ in the non-semisimple case, we cannot argue via topological field theory to obtain the multiplication.
Still, we will make $\lint^{X \in \Proj \cat{C}} \cat{C}(X,X)$ into a differential graded algebra below and prove that it is an algebra over the little disk operad $E_2$.
We see this $E_2$-algebra as a differential graded analogue of the Verlinde algebra. While the motivation in terms of topological field theory relies on modularity, the $E_2$-multiplication just needs the braiding. \\

Before stating the precise result, let us recall that the \emph{little disk operad $E_2$} is the topological operad whose space
$E_2(n)$ of arity $n$ operations is given by the space of affine embeddings from $n$ disks into another disk, see \cite[Chapter~4]{Fresse1} for details. This operad describes algebraic structures with a homotopy associative multiplication whose commutativity behavior is controlled by the braid group. 
By $ \Pi E_2$ we denote the operad in groupoids resulting from application of the fundamental groupoid functor $\Pi$ to $E_2$.

The space $E_2(n)$ is an Eilenberg-MacLane space $K(P_n , 1)$, where $P_n$ is the pure braid group on $n$ strands \cite[Chapter~5]{Fresse1}. We can also use an alternative description of $E_2(n)$ based on the short exact sequence
\begin{align} 0 \to P_n \to B_n \to \Sigma_n \to 0 \end{align} featuring besides the pure braid group $P_n$ also the braid group $B_n$ on $n$ strands and the permutation group $\Sigma_n$ on $n$ letters. The projection $B_n \to \Sigma_n$ defines a transitive action of $B_n$ on $\Sigma_n$,
and $\Pi E_2(n)$ is equivalent to the corresponding action groupoid,
\begin{align} \Pi E_2(n) \simeq \Sigma_n // B_n \ .      \label{eqnactiongroupoiddescription}    \end{align}
A permutation $\sigma \in \Sigma_n$ describes the affine embedding which aligns $n$ disks next to each other on the equator of a bigger disk with the order prescribed by $\sigma$.

As explained e.g.\ in \cite{salvatorewahl} or \cite[Chapter~5~and~6]{Fresse1}, algebras over $\Pi E_2$ are equivalent to braided monoidal categories; and in the description \eqref{eqnactiongroupoiddescription} of $\Pi E_2(2)$ we have the correspondences
\begin{align}
	\left\{\begin{array}{rcl}
		\text{identity permutation} & \xleftrightarrow{\ \text{\phantom{long}}\ }   & \text{tensor product}, \\ 
		\text{transposition of two letters} & \xleftrightarrow{\ \text{\phantom{long}}\ }   & \text{opposite tensor product}, \\ 
		\text{generator of $B_2$ braiding the two strands} & \xleftrightarrow{\ \text{\phantom{long}}\ }   & \text{braiding.} \\ 
	\end{array}\right\}\label{correspondenceeqn}
\end{align} 
Using the free functor $k[-]$ from sets to $k$-vector spaces we obtain from  $\Pi E_2$ an operad $k[\Pi E_2]$ in $k$-linear categories. Algebras over this operad are equivalent to $k$-linear braided monoidal categories.

There is a non-unital version of the $E_2$-operad that we call $\nE_2$. It has the same arity $n$ operations as $E_2$ for $n\ge 1$, but in arity zero, $\nE_2$ is empty unlike $E_2$ which is given by a point in arity zero. Categorical $\nE_2$-algebras are \emph{non-unital} braided monoidal categories, i.e.\ they do not necessarily have a monoidal unit.
We can now make the following elementary observation:

\begin{lemma}\label{lemmaaux1}
	For a braided finite  tensor category $\cat{C}$ over $k$, the subcategory $\Proj \cat{C} \subset \cat{C}$ of projective objects is a $k[\Pi \nE_2]$-algebra in $k$-linear categories. 
\end{lemma}

\begin{proof}
	Duality and exactness of the monoidal product ensure that the tensor product of two projective objects is again projective making $\Proj \cat{C}$ a $k$-linear non-unital monoidal category.
	If $\cat{C}$ is additionally braided, $\Proj \cat{C}$ is a $k$-linear non-unital braided monoidal  category, which proves the assertion.
\end{proof}

In order to obtain a differential graded version of the Verlinde algebra,
we combine the above Lemma with the following facts:

\begin{enumerate}[label={\normalfont(\arabic*)}]

	\item \label{prep1}
By definition,
for a $k$-linear category $\cat{D}$,
the homotopy coend $
\lint^{X\in\cat{D}} \cat{D}(X,X)$
is the realization of the simplicial vector space $\Loop\cat{D}$ which in degree $n$ is given by
\begin{align}
	\Loop_n \cat{D} = \bigoplus_{X_0,\dots,X_n \in \cat{D}} \cat{D}(X_1,X_0) \otimes \dots \otimes \cat{D}(X_n,X_{n-1}) \otimes \cat{D}(X_0,X_n) \ , \label{eqncatCs}
\end{align}
i.e.\ by the space of loops of morphisms in $\cat{D}$ through $n+1$ objects (as one finds by specializing Definition~\ref{defibar} to the case of a hom functor).
In other words, $\lint^{X\in\cat{D}} \cat{D}(X,X)$ is given by normalized chains on $\Loop \cat{D}$; in formulae  
\begin{align}
	\int_\mathbb{L}^{X \in \cat{D}} \cat{D}(X,X) \cong N_*(\Loop \cat{D})\label{eqncoendcsn} \ . 
\end{align}
The assignment $\cat{D} \mapsto \Loop \cat{D}$ yields a symmetric monoidal functor $\Cat_k \to \sVect_k$ from $k$-linear categories to simplicial $k$-vector spaces.

\item \label{prep2}

	For any symmetric monoidal bicategory $\cat{M}$, an $\cat{M}$-valued algebra over an $\cat{M}$-valued operad $\O$ can equivalently be described as a symmetric monoidal functor out of the symmetric monoidal category $F\O$ freely generated by that operad, see e.g.\ \cite[Section~1]{horel}. For a groupoid-valued operad like the non-unital $E_2$-operad $\Pi \nE_2$, we may see $F\Pi \nE_2$ actually as a symmetric monoidal bicategory. A $\Cat_k$-valued non-unital $E_2$-algebra may now be described as symmetric monoidal functor $F\Pi \nE_2\to\Cat_k$. 

\end{enumerate}

\begin{proposition}\label{propdgva}
	For every braided finite tensor category $\cat{C}$,
	the Hochschild complex $ \lint^{X \in \Proj \cat{C}} \cat{C}(X,X)$ is naturally a non-unital $E_2$-algebra in differential graded vector spaces. \end{proposition}

As the proof will show, neither rigidity nor finiteness of $\cat{C}$ are needed as long as $\Proj \cat{C}$ is still a non-unital $E_2$-algebra.

\begin{proof}
  By Lemma~\ref{lemmaaux1} and preparation~\ref{prep2} above $\Proj \cat{C}$ gives rise to a symmetric monoidal functor $F\Pi \nE_2\to\Cat_k$ that we can postcompose with the  symmetric monoidal functor $\Loop : \Cat_k\to\sVect_k$ from preparation~\ref{prep1}. The resulting symmetric monoidal functor $F\Pi \nE_2\to\sVect$ gives us, again by preparation~\ref{prep2}, the structure  of a non-unital differential graded $E_2$-algebra in simplicial vector spaces. After taking (normalized) chains, the assertion follows from~\eqref{eqncoendcsn}. 
	\end{proof}

In order to write down the product 
underlying the non-unital $E_2$-algebra, we need to establish some notation: Recall from \eqref{eqncatCs} that elements $\underline{f},\underline{g} \in \Loop_n \Proj \cat{C}$ are loops of $n+1$ morphisms between projective objects in $\cat{C}$. Using the monoidal product $\otimes$ of $\cat{C}$ we can tensor the morphisms of $\underline{f}$ and $\underline{g}$
together to obtain an element in $\Loop_n \Proj \cat{C}$ that we denote by $\underline{f} \otimes \underline{g}$. 
Next recall that for $0\le j\le n$
the degeneracy map $s_j : \Loop_n \Proj \cat{C} \to \Loop_{n+1} \Proj \cat{C}$ inserts the identity of the $j$-th object.
For a $(p,q)$-shuffle $(\mu,\nu)=(\mu_1,\dots,\mu_p,\nu_1,\dots,\nu_q)$, i.e.\ a permutation  of $\{1,\dots,p+q\}$ such that $\mu_1<\mu_2<\dots<\mu_p$ and $\nu_1<\nu_2<\dots<\nu_q$, we define the compositions
\begin{align}
	s_\mu := s_{\mu_p-1} \circ \dots \circ s_{\mu_1-1} \ ,\quad
	s_\nu := s_{\nu_q-1} \circ \dots \circ s_{\nu_1-1}
\end{align} of degeneracy maps.

\begin{corollary}\label{core2atchainlevel}
	For elements $\underline{f}$ and $\underline{g}$ of
	homological degree $p$ and $q$, respectively, the product 
	 from Proposition~\ref{propdgva} is explicitly given by 
	\begin{align}
		 \sum_{ \substack{ (p,q)\text{-shuffles} \ (\mu,\nu)   \\ \text{of $p+q$}}     }   \sign(\mu,\nu) \ \ s_\nu(\underline{f}) \otimes s_\mu(\underline{g})  \ . \label{eqnmultiplication}
	\end{align}
\end{corollary}

\begin{proof}
	By construction the $E_2$-multiplication on $\Loop \Proj \cat{C}$ comes from tensoring loops of morphisms together using the monoidal product of $\cat{C}$. 
	In order to obtain a formula for this multiplication on $\lint^{X \in \Proj \cat{C}} \cat{C}(X,X)$,
	we  use the structure maps of the symmetric lax monoidal functor $N_*$, namely the Eilenberg-Zilber maps \cite[8.5.4]{weibel}, and arrive at \eqref{eqnmultiplication}.
\end{proof}

\begin{remark}
	A finite tensor category is in particular rigid, and  the duality, when combined with the above methods, will lead to a 
	comultiplication on the Hochschild complex of a braided finite  tensor category
	(this corresponds to reading the bordism 
	$P\times \mathbb{S}^1:\mathbb{T}^2 \sqcup \mathbb{T}^2 \to \mathbb{T}^2$
	backwards).
A further investigation of the comultiplication, its relation to the multiplication
and  the higher structures they both give rise to
are beyond the scope of this article. 
\end{remark}

\begin{remark}[Boundary conditions and the Swiss-Cheese operad]
	Consider a braided finite tensor category $\cat{C}$ and a finite tensor category $\cat{W}$ together with a braided monoidal functor
	$
	F:\cat{C} \to Z( \cat{W})$.
	Such a structure appears in the description of the boundary condition in three-dimensional topological field theory \cite{fsv}.
	By one of the main results of \cite{idrissi}, this structure 
	precisely amounts to $(\cat{C},\cat{W},F)$ being a categorical algebra over the Swiss-Cheese operad introduced by Voronov \cite{voronov}. 
	A straightforward modification of Proposition~\ref{propdgva} shows us now 
	that the Hochschild chains $\lint^{X \in \Proj \cat{C}} \cat{C}(X,X)$ and $\lint^{Y \in \Proj \cat{W}} \cat{W}(Y,Y)$ of $\cat{C}$ and $\cat{W}$ and the map between those induced by $F$ form a differential graded Swiss-Cheese algebra.  
	By \cite{hoefel} the corresponding homology yields an algebraic structure closely related to the homotopy algebras used by Kajiura and Stasheff  \cite{ks} for the description of open-closed string field theories.
\end{remark}

\subsection{The equivariant case: Differential graded little bundles algebras\label{secequiv}}
Based on the discussion of the differential graded version of the Verlinde algebra 
and its motivation by topological field theory,
we can suggest, for a given finite group $G$, a reasonable candidate for a Hochschild complex of a  braided $G$-crossed monoidal category in the sense of Turaev
and exhibit an interesting multiplicative structure on it. 
We will first write down the candidate for the Hochschild complex 
and guess a multiplicative structure based on the ties of 
braided crossed monoidal categories to equivariant field theories. 
Then, we will turn this intuition into a precise statement using the little bundles operad defined in \cite{littlebundles} motivated by its relation to $(\infty,1)$-$G$-equivariant topological field theories \cite{MuellerWoikeHH}.

The notion of a
braided $G$-crossed monoidal category
is  based on 
\cite{turaevgcrossed,turaevhqft}. In the semisimple case, these categories are well-studied objects in equivariant representation theory
\cite{mueger,kirrilovg04,centerofgradedfusioncategories}. 
We follow the definition of \cite{galindo}, where in comparison to \cite{turaevhqft} more general coherence conditions are considered:
For a finite group $G$, a \emph{braided $G$-crossed category} is a $k$-linear category $\cat{C}$ that comes with a decomposition $\cat{C}=\bigoplus_{g \in G} \cat{C}_g$ and  is equipped with the following data:
\begin{itemize}
	\item A homotopy coherent action of $G$ on $\cat{C}$ making $h\in G$ act as an equivalence $\cat{C}_g \to \cat{C}_{hgh^{-1}}, X \mapsto h.X$.
	
	\item A $k$-linear monoidal product sending $\cat{C}_g \otimes \cat{C}_h$ to $\cat{C}_{gh}$. 
	
	\item A $G$-braiding consisting of natural isomorphisms
	\begin{align} X \otimes Y \cong g.Y \otimes X
	\end{align} for $X\in \cat{C}_g$ and $Y\in \cat{C}_h$ (this does \emph{not} yield a braiding on $\cat{C}$). 
\end{itemize}
For the details on the compatibilities and coherence requirements, we refer to \cite{galindo}, see also \cite{maiernikolausschweigerteq} and, additionally, \cite{littlebundles} for a description of braided $G$-crossed categories as algebras over the $G$-colored operad of parenthesized $G$-braids.
We define a \emph{braided finite  $G$-crossed tensor category} as a $k$-linear braided $G$-crossed monoidal category whose underlying $k$-linear monoidal category is a finite tensor category.

Categories of this type are intimately related to three-dimensional $G$-equivariant topological field theory \cite{turaevhqft,htv,hrt}, a flavor of topological field theory in which all manifolds are equipped with principal $G$-bundles. The decoration with principal $G$-bundles leads to interesting phenomena which are not present in the non-equivariant case.

Semisimple $G$-modular categories (a special type of  braided finite $G$-crossed tensor categories) are used in \cite{hrt} to construct a three-dimensional equivariant topological field theory. Conversely, given an extended three-dimensional $G$-equivariant topological field theory, its evaluation on the circle is a semisimple $G$-(multi)mo\-dular category \cite{extofk}. 
In the non-semisimple case, this interpretation of braided $G$-crossed tensor categories in terms of topological field theory breaks down as in the non-equivariant case.

Still, the perspective of topological field theory yields some tools 
for the study of non-semisimple  braided finite $G$-crossed tensor categories:
If $\cat{C}=\bigoplus_{g\in G} \cat{C}_g$ is a semisimple $G$-modular category, then the evaluation of the three-dimensional topological field theory built from $\cat{C}$ on the torus decorated with the principal bundle specified by the two commuting holonomies $g,z\in G$ is given by the coend
$
\int^{X \in \cat{C}_g} \cat{C}_g(z.X,X)
$
as explained in a different language in \cite[Section~VII.3]{turaevhqft} and worked out in terms of coends in \cite[Section~4.6]{extofk}. 
This suggests that in the non-semisimple case the collection of homotopy coends
\begin{align}
	\lint^{X \in \Proj \cat{C}_g} \cat{C}_g(z.X,X)     \label{eqnderivedcoendseq}
\end{align} provide a reasonable generalization of Hochschild chains to the equivariant case. 
More importantly, the topological intuition gives us an idea of the multiplicative structure that we should discover: Following the ideas laid out at the beginning of Section~\ref{secdgva}, crossing the pair of pants with a circle yields a bordism $\mathbb{T}^2  \sqcup \mathbb{T}^2 \to \mathbb{T}^2$.
In the equivariant case, this bordism has to be decorated with principal $G$-bundles. Upon fixing a central element $z\in Z(H)$, each pair of group elements $g_1,g_2\in G$ 
will provide the holonomies for a principal $G$-bundle on the bordism $\mathbb{T}^2  \sqcup \mathbb{T}^2 \to \mathbb{T}^2$ (for this we need $z$ to commute with $g_1$ and $g_2$); i.e.\ 
when denoting the bundle specified by the holonomies $z$ and some $g$ by $(z,g)$, we obtain 
a decorated bordism
\begin{align}  \left(  \mathbb{T}^2, (z,g_1)   \right)   \sqcup \left(  \mathbb{T}^2, (z,g_2)   \right)    \to (\mathbb{T}^2 , (z,    g_1 g_2     )) \ . \label{eqndecoratedbord}
\end{align} 
Note that we treat here the two $\mathbb{S}^1$-factors of the torus differently.
The fact that the holonomies $g_1$ and $g_2$ multiply is a consequence of the fundamental group of the pair of pants.
The decorated bordism \eqref{eqndecoratedbord} should give us a multiplication
\begin{align}
	\lint^{X \in \Proj \cat{C}_{g_1}} \cat{C}_{g_1}(z.X,X) \otimes \lint^{X \in \Proj \cat{C}_{g_1}} \cat{C}_{g_2}(z.X,X) \to \lint^{X \in \Proj \cat{C}_{g_1 g_2}} \cat{C}_{g_1 g_2}(z.X,X)      \label{eqnmultiintuition}
\end{align}
compatible with the group multiplication.
The commutativity behavior should be determined by the braid group action on the groupoid of principal $G$-bundles over the complement of little disk embeddings.
For instance, consider an embedding of two disks into a bigger disk and a principal bundle on the complement of this embedding. Such a principal bundle is determined by the holonomies $g$ and $h$ around the boundaries of the two embedded disks. When moving the disks past each other, these holonomies transform to $ghg^{-1}$ and $h$:
\begin{center}
	\centering\def\svgwidth{0.3\textwidth}
	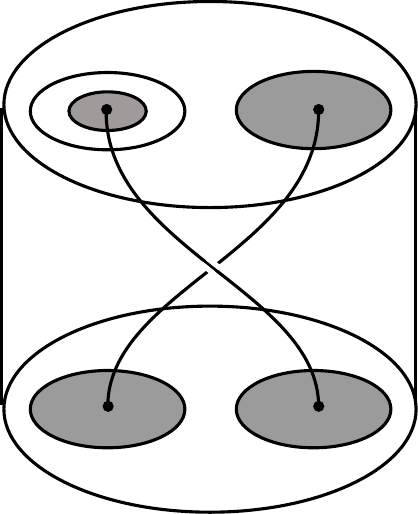
\end{center}
On the left disk, we have indicated the homotopy of classifying maps that acts as a gauge transformation
$h \xrightarrow{\ g \ } ghg^{-1}$. 
The multiplication \eqref{eqnmultiintuition} should have  a symmetry behavior reflecting the topological situation.

As the  main result of this subsection, we prove that indeed this topological intuition describes the multiplicative structure \eqref{eqnmultiintuition} accurately. 
We do this by showing that for a \emph{fixed}
central element $z\in Z(G)$ the assignment $	g \mapsto \lint^{X \in \Proj \cat{C}_g} \cat{C}_g(z.X,X)$ provides a (non-unital) algebra over the differential graded little bundles operad $E_2^G$ introduced in \cite{littlebundles}. 

The little bundles operad is an aspherical topological operad whose colors are the principal $G$-bundles over the circle (modeled as loops in $BG$) and whose operations $E_2^G \binom{\psi}{\underline{\varphi}}$ from a family $\underline{\varphi}=(\varphi_1,\dots,\varphi_n)$ to $\psi$ are given by affine embeddings $f \in E_2(n)$ equipped with a principal $G$-bundle on the complement of the image of 
$f$ restricting to the bundle $(\varphi_1,\dots,\varphi_n)$ on the $n$ inner boundary circles and to $\psi$ on the outer boundary circle.
The little bundles operad can be seen as an operad built from Hurwitz spaces, i.e.\ from the homotopy quotient of the braid group action on the moduli space of principal $G$-bundles over a punctured plane.

In \cite[Theorem~4.11 and 4.13]{littlebundles} $\Pi E_2^G$ is shown to be equivalent to the $G$-colored operad $\PBr^G$ of parenthesized $G$-braids whose categorical algebras are precisely braided $G$-crossed monoidal categories.  This turns the latter into a `topological object'.

Using the little bundles operad we can now make a precise statement about the family of complexes \eqref{eqnderivedcoendseq}:

\begin{proposition}\label{proplittlebundlesalgebra}
	Let $G$ be a finite group and $z\in Z(G)$ a fixed element in its center. Then for any braided finite  $G$-crossed tensor category $\cat{C}$,
	the assignment
	\begin{align}
		g \mapsto    \lint^{X \in \Proj \cat{C}_g} \cat{C}_g(z.X,X)
	\end{align} defines a non-unital
$E_2^G$-algebra in differential graded vector spaces, i.e.\ a non-unital differential graded little bundles algebra.
\end{proposition}

Using the operadic description of braided crossed categories from \cite{littlebundles} the proof is a straightforward generalization of the proof of Proposition~\ref{propdgva}.
	A crucial  step in the proof that is not present in the non-equivariant case is the observation that $z \in Z(G)$, seen as a unary little bundles operation, commutes with every little bundles operation $o : \cat{C}_{\underline{g}} \to \cat{C}_h$ in the sense that there is a canonical natural isomorphism $z.o(-)\cong o.(z.-)$. This follows from the fact that $G$ acts by monoidal functors on $\cat{C}$ and that $z$ lies in the center of  $G$.

\begin{remark}
	The homology $\bigoplus_{g\in G}   H_* \left(\lint^{X \in \Proj \cat{C}_g} \cat{C}_g(z.X,X)\right)$ of the little bundles algebra from Proposition~\ref{proplittlebundlesalgebra} is a graded algebra, but in contrast to the non-equivariant case, it is not graded commutative. Instead, for $x\in H_p\left(\lint^{X \in \Proj \cat{C}_g} \cat{C}_g(z.X,X)\right)$ and $y\in H_q\left(\lint^{X \in \Proj \cat{C}_h} \cat{C}_h(z.X,X)\right)$
	\begin{align} xy = (-1)^{pq}(g.y)x \ .\end{align}
\end{remark}

\spaceplease

\section{Homotopy coherent projective mapping class group action\label{secmcg}}
As the main result of this article, we establish a canonical homotopy coherent projective action of the mapping class group $\SL(2,\mathbb{Z})$ of the torus on the Hochschild complex of a modular category.
As explained in the introduction, our result provides a homotopy coherent extension of the work of \cite{lubamajid,lubacmp,luba,svea,shimizu}.

Let us briefly recall the definition of a modular category: 
For a braided finite tensor category, one defines the \emph{Müger center} as the subcategory spanned by all objects $X\in \cat{C}$ such that the double braiding $c_{Y,X}c_{X,Y}$
with every other object $Y \in \cat{C}$
 is the identity.
A braided finite tensor category is called \emph{non-degenerate} if its Müger center just consists of finite direct sums of the monoidal unit $I$,
see \cite{shimizumodular} for different characterizations of non-degeneracy.
A \emph{ribbon twist} on a braided finite tensor category $\cat{C}$ is a natural automorphism of the identity of $\cat{C}$ whose components $\theta_X:X \to X$ satisfy 
\begin{align}
\theta_{X\otimes Y} &= c_{Y,X}c_{X,Y}(\theta_X\otimes\theta_Y) \ , \\
\theta_I &= \id_I \ , \\
\theta_{X^\vee} &= \theta_X^\vee \ .
\end{align}
A \emph{finite ribbon category} is a braided finite tensor category equipped with a ribbon twist.
Finally, a \emph{modular category} is a finite ribbon category whose underlying braided finite tensor category is non-degenerate.
Since a modular category is ribbon and hence pivotal, we may use the techniques developed in the last section.

\subsection{Homotopy coherent projective actions\label{sechomcoh}}
We begin by recalling the notion of a homotopy coherent (projective) group action and by discussing suitable resolutions in order to write such actions down in the case of interest. The reader familiar with homotopy coherent actions can just skim through the lines below and take note of the specific resolutions that will be used. 

The idea underlying the notion of a homotopy coherent action $\varrho$ of a group $G$ on a chain complex $C$ is to relax the requirement that for $g,h\in G$ the chain maps $\varrho(gh)$ and $\varrho(g)\varrho(h)$ are equal. Instead, they will just be homotopic by a specific homotopy $H_{g,h} : \varrho(gh) \simeq \varrho(g)\varrho(h)$ that does not only exist, but is part of the data. 
Moreover, one requires these homotopies to be coherent, i.e.\ all the different homotopies $H_{g,h}$ for $g,h\in G$ should be related by higher homotopies: For example, for $g,h,\ell \in G$, the diagram
\begin{equation}
\begin{tikzcd}
& \varrho(g)\varrho(h\ell) \ar{rd}{\varrho(g)H_{h,\ell}}  & \\
\varrho(gh\ell)  \ar{ru}{H_{g,h\ell}} \ar{rd}[swap]{H_{gh,\ell}}  & \ & \varrho(g)\varrho(h)\varrho(\ell)\\
& \varrho(gh)\varrho(\ell)\ar[swap]{ru}{H_{g,h}\varrho(\ell)}
\end{tikzcd} 
\end{equation}
is required to commute up to homotopy, and again this homotopy is part of the data -- and so on and so forth,
as will be made precise below; 
for an introduction to homotopy coherent mathematics, we refer to \cite{riehl18}.

In the sequel, we will need a slight variation of the above, namely homotopy coherent \emph{projective} representations of a group $G$.
A projective $G$-representation on a vector space $V$ (or chain complex) 
is a group morphism $G \to \P\Aut(V)$, where $\P\Aut(V)$ is the quotient of $\Aut(V)$ by the normal subgroup $k^\times \cdot \id_V$. Often it is convenient to assign to each $g\in G$ an actual automorphism $\varrho(g)$ of $V$ by choosing a section of $\Aut(V) \to \P\Aut(V)$ as a map of sets. Then
\begin{align}
\varrho(g) \varrho(h) = \xi(g,h) \varrho(gh)
\end{align} for $g,h\in G$ and a cocycle $\xi \in Z^2(G;k^\times)$. We will then say that $\varrho$ is \emph{$\xi$-projective} because once a lift is chosen, the cocycle $\xi$ controls the projectivity.
This point of view is rather helpful because it allows us to describe projective actions via the \emph{twisted group algebra} $k_ \xi [G]$ of $G$ and $\xi \in Z^2(G;k^\times)$.
The underlying vector space of $k_\xi [G]$
is the free vector space on $G$. 
The multiplication is given by
\begin{align}
\langle g\rangle \langle h\rangle =  \xi(g,h) \langle gh\rangle \quad \text{for all}\quad g,h \in G \ ,
\end{align}
where $\langle g\rangle$ is the basis element corresponding to $g\in G$.
The cocycle $ \xi$ will be referred to as the \emph{twist}.
Now a $\xi$-projective $G$-representation on a vector space $V$ 
is just a $k_\xi[G]$-action on $V$.

In order to formalize the notion of a homotopy coherent (projective) action, we may as well define the notion of a homotopy coherent action of an algebra (which will then include the case of a twisted group algebra).
To this end, we recall the bar construction of an algebra $A$ over $k$ that one constructs via the free-forgetful adjunction
\begin{align}
\xymatrix{
	F \,:\, \Vect_k^{}       ~\ar@<0.5ex>[r]&\ar@<0.5ex>[l]  ~\Alg_k  \,:\, U \  .
}
\end{align} 
Suppressing the forgetful functor in the notation, we will see $F$ as an endo\-functor of $\Alg_k $. The algebra
$A$ gives rise to a simplicial algebra $\Bar A$ 
\begin{equation}
\begin{tikzcd}
\dots \ar[r, shift left=6]  \ar[r, shift left=2]
\ar[r, shift right=6]  \ar[r, shift right=2]
& 	F^3 A 
\ar[l, shift left=4]  \ar[l]
\ar[l, shift right=4]  
\ar[r, shift left=4] \ar[r, shift right=4] \ar[r] & F^2 A \ar[r, shift left=2] \ar[r, shift right=2]
\ar[l, shift left=2] \ar[l, shift right=2]
& FA\ ,  \ar[l] \\
\end{tikzcd}
\end{equation}
which in level $n$ is given by $F^{n+1} A$ -- the \emph{bar resolution of $A$}.
Of course, we can also see this simplicial algebra as a differential graded algebra
(via the Dold-Kan correspondence).
The bar construction comes with an augmentation $\Bar A \to A$ which is also a trivial fibration (on the level of underlying simplicial vector spaces, this augmentation admits extra degeneracies). 

The algebra of 0-simplices of $\Bar A$ is the free algebra on the vector space $A$, i.e.\ the tensor algebra on the vector space $A$. For an element $a_1 \otimes \dots \otimes a_n$ in the free algebra on $A$, we will write $(a_1) \dots (a_n)$. This bracket notation borrowed from \cite{riehl18}
is rather convenient because it allows  us 
to write the higher simplices of the bar construction by nested brackets. Then the $j$-th face operator $\partial_j$ deletes the $j$-th pair of brackets (counted from outside to inside). For example, for $a,b\in A$ we have a 1-simplex $((a)(b))$ with $\partial_0 ((a)(b))=(a)(b)$ and $\partial_1 ((a)(b))=(ab)$. Hence, the 0-simplices $(a)(b)$ and $(ab)$ are not equal, but there is a path between them.\label{bracketnotationpage}

In order to define homotopy coherent actions, we also need the internal hom of differential graded vector spaces: 
For differential graded vector spaces and $C$ and $D$,
 their internal hom $[C,D]$ is the differential graded vector space with $[C,D]_n := \prod_{m\in\mathbb{Z}} \Hom_k(C_m,D_{m+n})$. As usual,
 composition endows $[C,C]$ with the structure of a differential graded algebra.

\begin{definition}\label{defhomcohaction}
	For a $k$-algebra $A$, a \emph{homotopy coherent action of $A$ on a differential graded vector space $C$} is a map of differential graded algebras $\Bar A \to [C,C]$.
	\end{definition}

\begin{remark}
	For the reader familiar with homotopy coherent actions,
	 let us remark that this coincides with the usual definition of the homotopy coherent action of an operad \cite{bergermoerdijkbv} because $\Bar A$ is a cofibrant resolution of the operad whose unary operations are given by $A$. 
	\end{remark}

As a consequence, given a group $G$ and a cocycle $\xi \in Z^2(G;k^\times)$, a homotopy coherent $\xi$-projective action of $G$ on a differential graded vector space $C$ is a map $\Bar k_\xi[G] \to [C,C]$ of differential graded algebras.

\subsection{Homotopy coherent projective actions from central extensions of rank one}
Let $	0 \to J \to G \to H \to 0$ be a short exact sequence of groups. If we are given a representation of $G$ on a vector space, then it is easy to decide whether this representation descends to $H$: We just have to verify that all elements in the kernel $J$ of $G \to H$ are sent to the identity. A similar statement holds for projective actions.
If however we are given a (projective) action of $G$ on a chain complex and are able to show that all elements in the kernel $J$ of $G \to H$ act by chain maps which are homotopic to the identity, then this is not enough to conclude that we get \emph{in a canonical way}
 a homotopy coherent (projective) action of the quotient $H$.

The purpose of this subsection is to highlight at least one case, namely that of a central extension of rank one, in which we actually get a homotopy coherent (projective) action of the quotient. This result will be key for the construction of the homotopy coherent mapping class group action in the next subsection.

\begin{proposition}\label{propohtpactiondescend}
	Let 
	$
	0 \to \mathbb{Z} \to G \stackrel{\pi}{\to} H \to 0
	$ be a central extension of groups and $\xi \in Z^2(H;k^\times)$. We denote the image of $1\in\mathbb{Z}$ under $\mathbb{Z} \to G$ by $\tau$.
	Suppose we are given  a $\pi^*\xi$-projective representation  $\varrho$ of $G$ on a chain complex $C$ and a homotopy $\varrho(\tau)    \stackrel{L}{\simeq} \id_C$ such that $L\varrho(g)=\varrho(g)L$ for all $g\in G$. Then this data induces in a canonical way
	a homotopy coherent $\xi$-projective representation of $H$ on $C$.
\end{proposition}

The proof of this technical result will require the inductive construction of a simplicial map that will follow a standard procedure that we recall now:

\begin{lemma}\label{lemmaconstructionsimpmap}
	Let $X$ and $Y$ be simplicial vector spaces. Suppose we are given a family $\phi_n : X_n \to Y_n$ of linear maps that is characterized
	 inductively in the following way:
	\begin{xenumerate}
		
		\item $\phi_0 : X_0 \to Y_0$ is just an arbitrary linear map.
		
		\item For some $n\ge 1$, suppose the linear maps $\phi_p : X_p \to Y_p$ for $0\le p\le n-1$ are already  such that they satisfy $f^* \phi_q \sigma = \phi_p f^* \sigma$ for $f \in \Delta(p,q)$ and $\sigma \in X_q$ for $0\le p,q \le n-1$. 
		Furthermore,  the $\phi_n : X_n \to Y_n$ are given as follows: \begin{enumerate}
			\item[{\normalfont (a)}] If $\sigma \in X_n$ is degenerate, then there is a
			unique iterated degeneracy operator $S=s_{j_1} \dots s_{j_\ell}$ with $j_1>\dots>j_\ell$ and a unique non-degenerate $n-\ell$-simplex $\tau$ such that 
			 $\sigma = S \tau$.
			Then set $\phi_n(\sigma) = S \phi_{n-\ell}(\tau)$. 
			\item[{\normalfont (b)}] If $\sigma$ is non-degenerate, then denote by $\phi_{n-1}(\partial \sigma) : k[ \partial \Delta^n  ] \to Y$ the simplicial map that we obtain by evaluation of $\phi_{n-1}$ on the faces on $\sigma$ (the fact that this map is well-defined is due to the fact that the $\phi_p$ respect the simplicial operators for $0\le p\le n-1$). 
			Now  $\phi_n(\sigma)$ is given as a lift
			\begin{equation}
			\begin{tikzcd}[column sep=large, row sep=large]
			k[\partial \Delta^n] \ar[hookrightarrow]{d}\ar{rr}{    \phi_{n-1}(\partial \sigma)        } & &  Y   \\
			k[\Delta^n]  \ar[dashed,swap]{rru}{     \phi_n(\sigma)  } 
			\end{tikzcd},
			\end{equation}
			where the existence of such a lift is an assumption.
		\end{enumerate}
	\end{xenumerate}
	Then the maps $\phi_n : X_n \to Y_n$ are simplicial.
\end{lemma}

The proof of the Lemma is straightforward if one takes the statements \cite[Lemma~15.8.3\&4]{Hirschhorn} on de\-ge\-neracy operators into account.


Before we prove Proposition~\ref{propohtpactiondescend}, we need to introduce some notation:		Fix   a set-theoretic section $s : H \to G$ of $\pi$.
		The deviation of $s$ from being a group morphism is described by the classifying cocycle $\alpha \in Z^2(H;\mathbb{Z})$ of the central extension $0\to \mathbb{Z}\to G\to H\to 0$; i.e.\
		\begin{align} s h_1 sh_2 = \tau^{\alpha(h_1,h_2)} s(h_1h_2) \quad \text{for}\quad h_1,h_2 \in H \ . 
		\end{align} 
		For three elements in $H$, for example, we have
		\begin{align}
		s h_1 sh_2 sh_3 = \tau^{\alpha(h_1,h_2)} s(h_1h_2) sh_3 = \tau^{ \alpha(h_1,h_2) + \alpha(h_1h_2,h_3)   }   s(h_1h_2h_3)   
		\end{align}
		by applying the cocycle first to $h_1$ and $h_2$ and then to $h_1h_2$ and $h_3$. Alternatively, we could have applied it to $h_2$ and $h_3$ and then to $h_1$ and $h_2h_3$. The cocycle condition
		\begin{align} \alpha(h_1,h_2)+\alpha(h_1h_2,h_3) = \alpha(h_2,h_3)+\alpha(h_1,h_2h_3) \label{cocycleeqn}
		\end{align} tells us that both ways yield the same result.
		We will therefore denote any of the sides of \eqref{cocycleeqn} by $\alpha(h_1,h_2,h_3)$. More generally, we can define $\alpha(h_1,\dots,h_n)$ for $n$ elements in $H$ such that 
		\begin{align}
		s h_1 \dots  sh_n = \tau^{\alpha(h_1,\dots,h_n)}     s(h_1 \dots h_n) \ .  \label{cocycleeqn2}
		\end{align}
Such a multi-element notation will also be used for the cocycle $\xi \in Z^2(H;k^\times)$ describing the projectivity.

\begin{proof}[\textsl{Proof of Proposition~\ref{propohtpactiondescend}}]
	Following Definition~\ref{defhomcohaction} we have to build a map $\varphi : \Bar k_\xi [H] \to [C,C]$ of differential graded algebras. We will equivalently describe it as a map of simplicial algebras and construct the underlying simplicial map following Lemma~\ref{lemmaconstructionsimpmap}. 
	Additionally, we will make sure that in every degree the algebra structure is respected such that we actually obtain a map of simplicial algebras.

First we set $\varphi(h)=\varrho(sh)$ for $h\in H$ and our fixed set-theoretic section $s : H \to G$
(up to homotopy, the construction will not depend on the choice of the section).
 This assignment extends to an algebra map
		\begin{align}
	\Bar_0 k_\xi [H]=F k[H] \to \Ch(C,C) 
		\end{align}
		because $\Bar_0 k_\xi[H]$ is freely generated as an algebra by the elements of $H$.
		This way,  we obtain the definition of $\varphi$ on $0$-simplices.
		
		Next, we consider a 1-simplex $\sigma \in \Bar_1 k_\xi[H]=F^2 k[H]$ of the form $\sigma =( (h_1) \dots (h_n)  )$, where we use the bracket notation explained on page~\pageref{bracketnotationpage}.
		These freely generate $F^2 k[H]$ as an algebra.
		We can see $\sigma$ as a path in the bar construction $\Bar k_\xi [H]$  from
		$\partial_0 \sigma = (h_1) \dots (h_n)$ to $\partial_1 \sigma = \xi(h_1,\dots,h_n) (h_1 \dots h_r)$, where $\xi(h_1,\dots,h_n)$
		is the multi-element notation for the cocycle $\xi$ just introduced. Therefore, we depict the 1-simplex $\sigma$ by
			\begin{equation}
					\centering
					\includegraphics[width=0.3\textwidth]{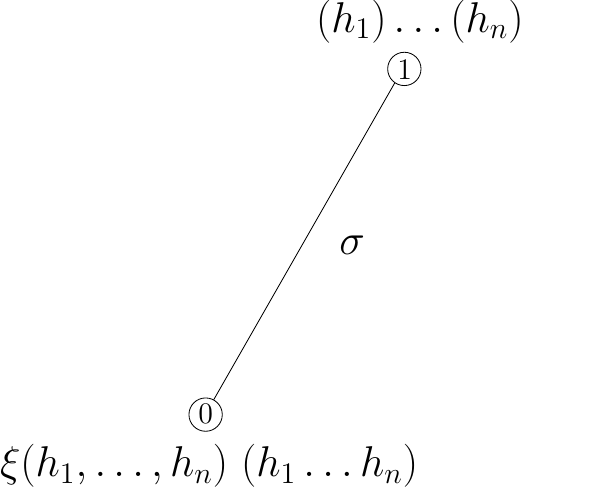}.
				\end{equation}
				If $\sigma$ is degenerate, then Lemma~\ref{lemmaconstructionsimpmap} tells us how to define $\varphi$ on it. Suppose now that $\sigma$ is non-degenerate.
	By definition of $\varphi$ on 0-simplices, we have
		\begin{align}
		\varphi ( \partial_0 \sigma) &= \varphi(h_1) \dots \varphi(h_n) = \varrho(sh_1) \dots \varrho(sh_n) \\&= \xi(h_1,\dots,h_n) \varrho(sh_1 \dots sh_n)\\&=\xi(h_1,\dots,h_n)\varrho(\tau)^{\alpha(h_1,\dots,h_n)} \varrho(s(h_1 \dots h_n)) \ ,\quad  \text{see} \ \eqref{cocycleeqn2}\ ;\\ 
		\varphi(\partial_1 \sigma) &= \xi(h_1,\dots,h_n) \varrho(  s(h_1 \dots h_n)  ) \ . 
		\end{align}
		In summary, we arrive at
		\begin{align}
		\varphi ( \partial_0 \sigma) = \varrho(\tau)^{\alpha(h_1,\dots,h_n)} \varphi(\partial_1 \sigma) \ . 
		\end{align} Now we can assign to $\sigma$ the homotopy $L^{\alpha(h_1,\dots,h_n)} \varrho(\partial_1 \sigma)$, i.e.\ the $\alpha(h_1,\dots,h_n)$-th power of $L$ combined with the identity homotopy of $\varrho(\partial_1 \sigma)$.
		Since
		$L^{\alpha(h_1,\dots,h_n)} \varrho(\partial_1 \sigma)=\varrho(\partial_1 \sigma) L^{\alpha(h_1,\dots,h_n)}$
		by assumption, we will suppress the identity homotopy in the notation.
This assignment extends to an algebra map
\begin{align}
\Bar_1 k_\xi[H] = F^2 k[H] \to \Ch(C\otimes N_*(\Delta^1;k), C) \ 
\end{align}
because the 1-simplices that we considered freely generate $\Bar_1 k_\xi[H]$ as an algebra.
This concludes the definition of $\varphi$ on 1-simplices.

		Consider now a 2-simplex $\sigma \in \Bar_2 k_\xi[H]=F^3 k[H]$ of the form
		\begin{align}
		\sigma = \left(   \left(  (h_{1,1}) \dots (h_{1,m_1})   \right) \dots \left(  (h_{n,1}) \dots (h_{n,m_n})   \right)  \right) \ ;
		\end{align}
		again these freely generate $\Bar_2 k_\xi[H]$ as an algebra. We will now prove that the homotopies that we obtain by evaluation of $\varphi$ on the boundary of $\sigma$ provide a strictly commuting triangle. This allows us to define $\varphi$ on $\sigma$ as the identity 2-homotopy. Then we extend multiplicatively and hence obtain the definition of $\varphi$ on 2-simplices. This way we still follow the construction principle from Lemma~\ref{lemmaconstructionsimpmap} and make sure that $\varphi$ respects the multiplicative structure. Since to the 2-simplices we have just assigned the identity, the definition on higher simplices can be completed trivially by assigning again identities, thereby completing the definition of $\varphi$ as a map of simplicial algebras. 
		
		It still remains to show that the homotopies assigned by $\varphi$ to the boundary of $\sigma$ form a strictly commuting triangle.
		To this end, we depict the faces and vertices of $\sigma$ as follows: 
		\begin{equation}
			\centering
			\includegraphics[width=0.9\textwidth]{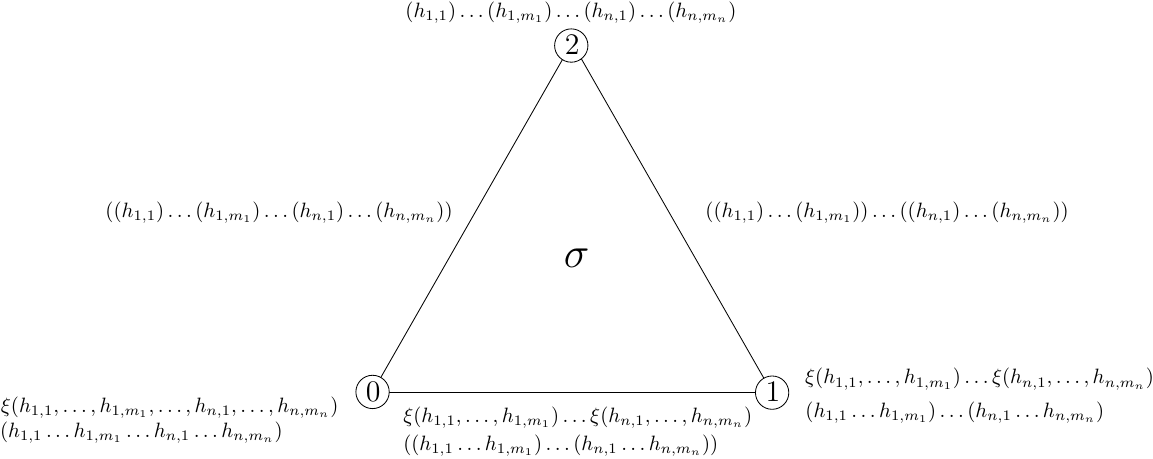}
		\end{equation}
		By the definition of $\varphi$ on 0- and 1-simplices the image of the boundary of $\sigma$
		under $\varphi$ is given by:
				\begin{equation}\label{eqnstrictlycommtriangle}
					\centering
					\includegraphics[width=0.9\textwidth]{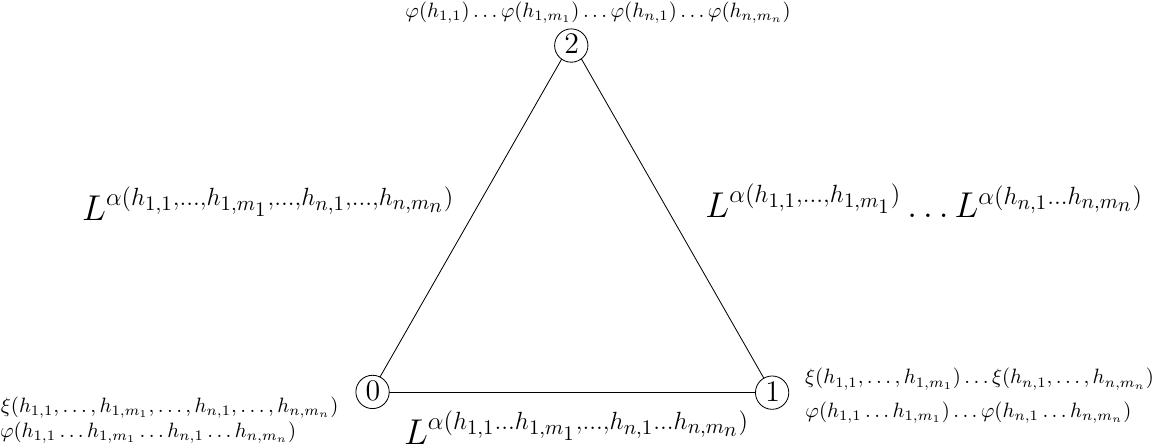}
				\end{equation}
				Since $L$ commutes with $\varphi$ by assumption, we are allowed to omit the identity homotopies in the notation.
		By the cocycle condition on $\alpha$, we have
		\begin{align}
		\alpha( h_{1,1}, \dots ,h_{1,m_1},  \dots , h_{n,1}, \dots ,h_{n,m_n}) = &\alpha( h_{1,1} \dots h_{1,m_1},  \dots , h_{n,1}\dots h_{n,m_n})\\
		&+ \alpha ( h_{1,1}, \dots, h_{1,m_1}) +\dots+\alpha ( h_{n,1}, \dots, h_{n,m_n})\ , 
		\end{align}
		which allows us to conclude that \eqref{eqnstrictlycommtriangle} forms a strictly commutative triangle.
		This finishes the proof.
\end{proof}

\subsection{Homotopy coherent projective $\SL(2,\mathbb{Z})$-action on the Hochschild complex of a modular category}
Having discussed the notion of a homotopy coherent projective action and some tools for its construction, we now finally exhibit a homotopy coherent projective action of the mapping class group of the torus on the Hochschild complex of a modular category.

The central non-homotopical ingredient will be the Lyubashenko-Majid action on the canonical coend:
Let $\cat{C}$ be a modular category and $\mathbb{F}$ the canonical (ordinary) coend. As one of the main results of \cite{lubamajid,lubacmp,luba}, Lyubashenko and Majid construct a projective action of the mapping class group of the punctured torus, i.e.\ the braid group on three strands $B_3 = \langle s,t,r \, | \, (st)^3=s^2 , s^4 =r \rangle$,
 on $\mathbb{F}$. 
 These authors give explicit automorphisms of $\mathbb{F}$ for each generator such that $r$ is sent to the inverse twist $\theta_\mathbb{F}^{-1}$ of $\mathbb{F}$.
\label{pagerefribbontwistaction} The $B_3$-action on $\mathbb{F}$ descends to an action of the mapping class group of the torus $\SL(2,\mathbb{Z}) = \langle s,t \, | \, (st)^3=s^2 , s^4 =1\rangle$ on $\cat{C}(I,\mathbb{F})$,
  as follows from the naturality of the twist and $\theta_I=\id_I$. 
  
  \begin{remark}\label{remproj}
  The projective $B_3$-action on $\mathbb{F}$ can be turned into a linear one since $H^2(B_3;k^\times)=0$. However, then the generator $r$ might not be sent to the inverse twist any longer. 
  Therefore, we refrain from getting rid of the projectivity.
  \end{remark}

Using 
the following Lemma the cocycle description of the projectivity of the representation of $B_3$ on $\mathbb{F}$ can be 
simplified.

\begin{lemma}\label{lemmapulbackcocycle}
	Let $\pi : G \to H$ an epimorphism of groups and $X$ and $Y$ objects in a $k$-linear category such that $\cat{C}(Y,X)\neq 0$.
	Then any projective representation $\varphi :G \to \P \Aut X$ of $G$ on $X$ 
	for which the composition
	\begin{align}
	G \xrightarrow{\ \varphi\ } \P \Aut X \to \P \Aut \cat{C}(Y,X)
	\end{align} is trivial on $\ker \pi$ 
	can be described by a family of automorphisms $\varrho(g)$ for $g\in G$ such that the cocycle controlling the projectivity with respect to these maps is the pullback $\pi^* \xi$ of a cocycle $\xi \in Z^2(H;k^\times)$ along $\pi$. 
\end{lemma}

\begin{proof}
	The projective representation of $G$ on $X$ is a group morphism $\varphi : G \to \P \Aut X$. Its concatenation $\psi := \cat{C}(Y,-)\circ \varphi  : G \to \P \Aut \cat{C}(Y,X)$ with $\cat{C}(Y,-)$ sends $\ker \pi$ to $1$. Now we choose a lift 
	for each element in $\P \Aut \cat{C}(Y,X)$ that sends the unit to the identity of $\cat{C}(Y,X)$ and denote the lift 
	of $\psi(g)$ for $g\in G$ by $\widetilde \psi(g)$.
	If $\widehat \varphi(g)$ is \emph{any} lift of $\varphi(g)$, then $\cat{C}(Y,-)$ maps $\widehat \varphi(g)$ to $c_g \widetilde \psi(g)$ for a unique invertible scalar $c_g\in k^\times$. We now set $\widetilde \varphi(g) := \widehat \varphi(g)/c_g$ and thereby ensure that $\cat{C}(Y,-)$ maps $\widetilde \varphi(g)$ to $\widetilde \psi(g)$.
	If we let $\nu \in Z^2(G;k^\times)$ be the cocycle describing the projectivity of $\varphi$ with respect to the representatives $\widetilde \varphi(g)$, we find by definition
	\begin{align}
	\widetilde \varphi(g) \widetilde \varphi(g') = \nu(g,g') \widetilde \varphi(gg')
	\end{align} for $g,g'\in G$. When applying the functor $\cat{C}(Y,-)$ in the special case $g' \in \ker \pi$, we obtain
	\begin{align}
	\widetilde \psi(g) = \widetilde \psi(g) \widetilde \psi(g') = \nu(g,g') \widetilde \psi(gg')=\nu(g,g') \widetilde \psi(g)
	\end{align} by the choice of our lifts and the assumption that $\psi$ sends $\ker \pi$ to the unit.
	Since $\cat{C}(Y,X)\neq 0$, we conclude $\nu(g,g')=1$. Hence, $\nu$ is trivial on $G\times \ker \pi$ and similarly on $\ker \pi \times G$. Now a direct computation shows that $\xi(h,h'):=\nu(s(h),s(h'))$ for $h,h'\in H$ and any set-theoretic section $s:H\to G$ of $\pi$ defines a 2-cocycle $\xi$
	 on $H$ with $\pi^* \xi=\nu$.
\end{proof}

If we apply this Lemma to the epimorphism $B_3 \to \SL(2,\mathbb{Z})$, $X=\mathbb{F}$ and $Y=I$, then thanks to $\cat{C}(I,\mathbb{F})\neq 0$ and the fact that the $B_3$-action descends to an $\SL(2,\mathbb{Z})$-action on $\cat{C}(I,\mathbb{F})$, we conclude that the projectivity of the $B_3$-action on $\mathbb{F}$ can be described 
by the pullback of a cocycle $\xi \in Z^2(\SL(2,\mathbb{Z});k^\times)$.
By looking at the proof of Lemma~\ref{lemmapulbackcocycle} 
we see that we can still arrange 
that the generator $r$ 
of the kernel of the projection $B_3 \to \SL(2,\mathbb{Z})$
 is sent to 
 $\theta_\mathbb{F}^{-1}$.

Let us now state the main Theorem:

\begin{theorem}\label{thmsl2z}
	The Hochschild complex $\lint^{X \in \Proj \cat{C}} \cat{C}(X,X)$ of a modular category $\cat{C}$
	carries a homotopy coherent projective action of the mapping class group $\SL(2,\mathbb{Z})$ of the torus which is induced in a canonical way by the action of the braid group on three strands on the canonical coend of $\cat{C}$.
	\end{theorem}

\begin{remark}	\label{remprojthm}We may actually include a statement about the concrete nature of the projectivity of this action: Using Lemma~\ref{lemmapulbackcocycle} we have established that the projectivity of the $B_3$-action on $\mathbb{F}$ is controlled by the pullback of a cocycle $\xi \in Z^2(\SL(2,\mathbb{Z});k^\times)$. The homotopy coherent projective $\SL(2,\mathbb{Z})$-action that we construct will be an action of the bar resolution of the $\xi$-twisted group algebra of $\SL(2,\mathbb{Z})$. 
 \end{remark}
 
 Before going into the proof of Theorem~\ref{thmsl2z}, we discuss one immediate consequence following from the fact that taking finite-dimensional modules over a ribbon factorizable Hopf algebra provides an example of a modular category:
 
 \begin{corollary}\label{corhopfalg}
 There is a canonical homotopy coherent projective $\SL(2,\mathbb{Z})$-action on the Hochschild complex of a ribbon factorizable Hopf algebra.
 \end{corollary}
 
 When stated in this form, our result is a direct homotopy coherent extension of \cite{svea}.\\[0.5ex]

 For the proof of Theorem~\ref{thmsl2z} we will need a few further technical results. One key step will be to replace the Hochschild complex by an equivalent complex. 
 
 \begin{proposition}\label{propalternativcomplex}
 For any pivotal finite tensor category $\cat{C}$,
  the Hochschild complex of $\cat{C}$ is canonically equivalent to the homotopy coend $\lint^{P\in\Proj \cat{C }}   \cat{C}(I,P) \otimes \cat{C}(P,\mathbb{F})$.
 \end{proposition}

 The proof of this Proposition will need the following Lemma:
 
 \begin{lemma}\label{lemyeondaprojective}
 	Let $\cat{C}$ be a finite category. Then for any $X,Y\in\cat{C}$ there is a natural equivalence
 	\begin{align}
 	\lint^{P\in\Proj \cat{C }}   \cat{C}(Y,P) \otimes \cat{C}(P,X) \simeq \cat{C}(Y,\cof X) \ . 
 	\end{align}
 \end{lemma}
 
 \begin{proof}
 	Since $\cat{C}(P,X) \simeq \cat{C}(P,\cof X)$ for any $P\in \Proj \cat{C}$, it suffices to prove that 
 	the natural map
 	\begin{align}	\lint^{P\in\Proj \cat{C }}   \cat{C}(Y,P) \otimes \cat{C}(P,\cof X) \to \cat{C}(Y,\cof X)   \label{eqnyonedaexactenqn} \end{align}
 	is an equivalence.
 	For this we realize that the left hand side is the total complex of the first quadrant double complex given in degree $(m,n)$ by
 	\begin{align}
 	\bigoplus_{P_0,\dots,P_m \in \Proj \cat{C}}          \cat{C}(P_1,P_0) \otimes \dots \otimes \cat{C}(P_{m},P_{m-1}) \otimes \cat{C}(  P_0,\cof_n X   ) \otimes \cat{C}(Y,P_m) \ .     \label{eqndoublecx}
 	\end{align}
 	Hence, the $n$-th row is $\lint^{P \in \Proj \cat{C}} \cat{C}(P,\cof_n X) \otimes \cat{C}(Y,P)$; and since $\cof_n X$ is projective, this is 
 	equivalent to $\cat{C}(Y,\cof_n X)$ by the Yoneda Lemma~\ref{propyoneda} applied to $\Proj \cat{C}$. Now the spectral sequence associated to the filtration of the double complex \eqref{eqndoublecx} by rows shows that \eqref{eqnyonedaexactenqn} is an equivalence. 
 \end{proof}
 
 \begin{proof}[\textsl{Proof of Proposition~\ref{propalternativcomplex}}]
 We have established in Theorem~\ref{theoderivedcoendviaobject} that the Hochschild complex of $\cat{C}$ is canonically equivalent to $\cat{C}(I,\cof \mathbb{F})$, where $\cof \mathbb{F}$ is a projective resolution of $\mathbb{F}$. Now the assertion follows from Lemma~\ref{lemyeondaprojective}.
 \end{proof}

The idea for the proof of Theorem~\ref{thmsl2z} is to apply Proposition~\ref{propohtpactiondescend} to the $B_3$-action on the complex $\lint^{P\in\Proj \cat{C }}   \cat{C}(I,P) \otimes \cat{C}(P,\mathbb{F})$ induced by the $B_3$-action on $\mathbb{F}$. The following statement asserts that the assumptions of Proposition~\ref{propohtpactiondescend} are met:

\begin{proposition}\label{propcommutinghomotopy}
	Let $\cat{C}$ be a finite ribbon category and $X \in \cat{C}$.
	Consider the action of  the ribbon twist $\theta_X$ of $X$  via postcomposition on the complex $\lint^{P\in\Proj \cat{C }}   \cat{C}(I,P) \otimes \cat{C}(P,X)$. The corresponding chain map ${\theta_X}_*$ is homotopic to the identity via a canonical homotopy $h$.
	If $f: X \to X$ is any endomorphism of $X$, then $f_* h=hf_*$ for the chain map $f_*$ that $f$ gives rise to.
\end{proposition}

\begin{proof}
	We treat $\lint^{P\in\Proj \cat{C }}   \cat{C}(I,P) \otimes \cat{C}(P,X)$ as a simplicial vector space and construct $h$ as a simplicial homotopy.
	An $n$-simplex of $\lint^{P\in\Proj \cat{C }}   \cat{C}(I,P) \otimes \cat{C}(P,X)$ is a string
	\begin{align}
	\underline{f} = \left(   I \xrightarrow{\  f_{n}  \ } P_n \xrightarrow{\  f_{n-1}  \ } P_{n-1} \xrightarrow{\  f_{n-2}  \ } \dots \xrightarrow{\  f_{0}  \ } P_0 \xrightarrow{\  f_{-1}  \ } X \right) \label{genericstringeqn}
	\end{align}
	of morphisms in $\cat{C}$, where the objects $P_j$ for $0\le j\le n$ are projective (it lives in the summand indexed by $P_0,\dots,P_n$). For $0\le j \le n$, we define $h_j \underline{f}$ as the $n+1$-simplex
	\begin{align}
	I \xrightarrow{\  f_{n}  \ } P_n \xrightarrow{\  f_{n-1}  \ } P_{n-1} \xrightarrow{\  f_{n-2}  \ } \dots  \xrightarrow{\  f_{j}  \ } P_j \xrightarrow{\  \theta_{P_j}  \ } P_j \xrightarrow{\  f_{j-1}  \ } \dots   \xrightarrow{\  f_{0}  \ } P_0 \xrightarrow{\  f_{-1}  \ } X \ , 
	\end{align}
	i.e.\ $h_j$ inserts the ribbon twist of $P_j$.
	A direct computation
	using the naturality of the twist and $\theta_I=\id_I$
	 shows that
the maps \begin{align} h_j : \left( \lint^{P\in\Proj \cat{C }}   \cat{C}(I,P) \otimes \cat{C}(P,X)  \right)_n \to    \left( \lint^{P\in\Proj \cat{C }}   \cat{C}(I,P) \otimes \cat{C}(P,X)  \right)_{n+1}\end{align} form a simplicial homotopy   \cite[Definition~8.3.11]{weibel}
between ${\theta_X}_*$ and the identity of $\lint^{P\in\Proj \cat{C }}   \cat{C}(I,P) \otimes \cat{C}(P,X)$. 
For any endomorphism $f: X \to X$ we see $h_jf_*=f_*h_j$. The same is true for the chain homotopy that the $h_j$ give rise to.
\end{proof}

We can now tie all the technical results together:

\begin{proof}[\slshape Proof of Theorem~\ref{thmsl2z}]
	Recall that we can describe the ordinary  projective action of $B_3$ on $\mathbb{F}$ as a $\pi^* \xi$-projective action for a cocycle $\xi \in Z^2(\SL(2,\mathbb{Z});k^\times)$, where $\pi : B_3 \to \SL(2,\mathbb{Z})$ is the canonical projection (this was a consequence of Lemma~\ref{lemmapulbackcocycle}).
	By Proposition~\ref{propalternativcomplex} we know that the Hochschild complex $\lint^{X\in\Proj \cat{C}} \cat{C}(X,X)$ of $\cat{C}$ is equivalent to the homotopy coend $\lint^{P\in\Proj \cat{C }}   \cat{C}(I,P) \otimes \cat{C}(P,\mathbb{F})$. Hence, we may as well exhibit a homotopy coherent projective $\SL(2,\mathbb{Z})$-action on $\lint^{P\in\Proj \cat{C }}   \cat{C}(I,P) \otimes \cat{C}(P,\mathbb{F})$.

 To this end, we note that by postcomposition we obtain a $\pi^* \xi$-projective action of $B_3$ on   $\lint^{P\in\Proj \cat{C }}   \cat{C}(I,P) \otimes \cat{C}(P,\mathbb{F})$. The strategy is now to apply Proposition~\ref{propohtpactiondescend} to this projective $B_3$-action and the central extension
	 $0\to \mathbb{Z}\to B_3 \stackrel{\pi}{\to} \SL(2,\mathbb{Z})\to 0$ to conclude that the projective $B_3$-action descends to a homotopy coherent projective $\SL(2,\mathbb{Z})$-action. Then from the proof of Proposition~\ref{propohtpactiondescend} we can read off that this will be constructed as an action of the bar resolution of the $\pi^*\xi$-twisted group algebra of $\SL(2,\mathbb{Z})$, which justifies the additional statement on the projectivity of the action that we have included in Remark~\ref{remprojthm}.
	 
	 In order to apply Proposition~\ref{propohtpactiondescend},
	  it remains to prove 
	 the generator $r$ of $B_3$ acts by a map which is homotopic to the identity by a chain homotopy that commutes with all chain maps that constitute the projective $B_3$-action on $\lint^{P\in\Proj \cat{C }}   \cat{C}(I,P) \otimes \cat{C}(P,\mathbb{F})$.
	But $r$ acts on $\lint^{P\in\Proj \cat{C }}   \cat{C}(I,P) \otimes \cat{C}(P,\mathbb{F})$ by postcomposition with the inverse ribbon twist of $\mathbb{F}$, as follows from the properties of the Lyubashenko-Majid action recalled on page~\pageref{pagerefribbontwistaction}. Hence, the desired statement can be deduced from Proposition~\ref{propcommutinghomotopy}.
	\end{proof}

\begin{remark}\label{remarkfactorization}
The fact that in the proof of Theorem~\ref{thmsl2z} we use the complex $\lint^{P\in\Proj \cat{C }}   \cat{C}(I,P) \otimes \cat{C}(P,\mathbb{F})$, exhibit a projective $B_3$-action on it and prove that it descends up to coherent homotopy to $\SL(2,\mathbb{Z})$ is not by accident:
In the language of conformal field theory, $\cat{C}(P,\mathbb{F})$ is the conformal block for torus with one boundary disk labeled by $P$. If we think of the Hochschild complex $\lint^{X \in \Proj \cat{C}} \cat{C}(X,X)$ as the differential graded conformal block for the torus, then one can formulate a differential graded factorization property for the gluing of a disk to the torus with one boundary circle (it is crucial that the gluing is implemented via a homotopy coend over projective objects). In this process, the $B_3$-action descends up to coherent homotopy to a $\SL(2,\mathbb{Z})$-action, i.e.\ from the mapping class group of the torus with one boundary circle to the mapping class group of the closed torus -- just as one would expect. 
We will not make precise the idea of differential graded conformal blocks here, but still the heuristics just presented make the strategy in the proof of Theorem~\ref{thmsl2z} more transparent. 
A version of this reasoning without homotopy coherence, i.e.\ on the level of homology, is used in \cite{svea2}.
\end{remark}

\begin{remark}[Modular homology]\label{remmodularhomology}
	It is natural to ask about the homotopy orbits of the homotopy coherent projective mapping class group action on the Hochschild complex of a modular category that we obtain from Theorem~\ref{thmsl2z}. This leads to a complex whose homology one might call the \emph{modular homology} of a modular category. Unlike Hochschild homology, this new algebraic invariant should be sensitive to more structure of the modular category than just the underlying linear category, e.g.\ the braiding and the ribbon twist. 
For Drinfeld doubles, the modular homology relates to the mapping class group orbits of bundles.	A detailed investigation  is beyond the scope of this paper and will be the subject of future work.
	\end{remark}


\end{document}

%% file: Fig_braidingBW.pdf_tex
\begingroup%
  \makeatletter%
  \providecommand\color[2][]{%
    \errmessage{(Inkscape) Color is used for the text in Inkscape, but the package 'color.sty' is not loaded}%
    \renewcommand\color[2][]{}%
  }%
  \providecommand\transparent[1]{%
    \errmessage{(Inkscape) Transparency is used (non-zero) for the text in Inkscape, but the package 'transparent.sty' is not loaded}%
    \renewcommand\transparent[1]{}%
  }%
  \providecommand\rotatebox[2]{#2}%
  \ifx\svgwidth\undefined%
    \setlength{\unitlength}{200.43476563bp}%
    \ifx\svgscale\undefined%
      \relax%
    \else%
      \setlength{\unitlength}{\unitlength * \real{\svgscale}}%
    \fi%
  \else%
    \setlength{\unitlength}{\svgwidth}%
  \fi%
  \global\let\svgwidth\undefined%
  \global\let\svgscale\undefined%
  \makeatother%
  \begin{picture}(1,1.23015191)%
    \put(0,0){\includegraphics[width=\unitlength]{Fig_braidingBW.pdf}}%
    \put(0.66380223,1.095){\color[rgb]{0,0,0}\makebox(0,0)[lb]{\smash{$g$}}}%
    \put(0.23105421,1.095){\color[rgb]{0,0,0}\makebox(0,0)[lb]{\smash{$ghg^{-1}$}}}%
    \put(0.346,0.95){\color[rgb]{0,0,0}\makebox(0,0)[lb]{\smash{$\xrightarrow{ \, g\,\, }$}}}%
    \put(0.34481308,0.081){\color[rgb]{0,0,0}\makebox(0,0)[lb]{\smash{$g$}}}%
    \put(0.63907761,0.081){\color[rgb]{0,0,0}\makebox(0,0)[lb]{\smash{$h$}}}%
  \end{picture}%
\endgroup%

%% file: hhftc.bbl
\begin{thebibliography}{1000000000}\itemsep0pt
	\bibitem[BK00]{bakifm}
B. Bakalov, A. Kirillov.
On the Lego-Teichmüller game.
\emph{Transformation Groups} 5:207--244, 2000.

\bibitem[BK01]{baki}
B. Bakalov, A. Kirillov.
 \emph{Lectures on tensor categories and modular functors.}
 University Lecture Series 21,
 Am. Math. Soc.,  Providence, 2001.


\bibitem[BDSPV15]{BDSPV15}
B. Bartlett, C. L. Douglas, C. J. Schommer-Pries, J. Vicary.
{Modular Categories as Representations of the 3-dimensional Bordism 2-Category}. 
 arXiv:1509.06811 [math.AT]












	\bibitem[BM06]{bergermoerdijkbv}	
C. Berger, I. Moerdijk. {The Boardman-Vogt resolution of operads in monoidal model categories.} \emph{Topology} 45:807--849, 2006.














\bibitem[CR05]{cibils}
C. Cibils,  M. J. Redondo. { Cartan-Leray spectral sequence
for Galois coverings of linear categories.} \emph{J. Algebra} 284:310--325, 2005.


		\bibitem[DW90]{dijkgraafwitten}
R. Dijkgraaf, E. Witten.
{Topological Gauge Theories and Group Cohomology}.
\emph{Commun. Math. Phys.} 129:393-429, 1990.



\bibitem[DSPS19]{dss}
C. L. Douglas, C. Schommer-Pries, N. Snyder. The balanced tensor product of module categories.
\emph{Kyoto J.  Math.} 59(1):167--179, 2019.



\bibitem[EGNO15]{egno}
P. Etingof, S. Gelaki, D. Nikshych, V. Ostrik.
\emph{Tensor categories.}
Math. Surveys  Monogr.  205,
Am. Math. Soc.,  Providence, 2015.


\bibitem[EO04]{etinghofostrik}
P. Etingof, V. Ostrik. {Finite tensor categories.} 
\emph{Mosc. Math. J.}
4(3):627--654, 2004.

	\bibitem[FQ93]{freedquinn}
D. S. Freed, F. Quinn.
{Chern-Simons Theory with Finite Gauge Group}.
\emph{Commun. Math. Phys.} 156:435--472, 1993.

\bibitem[Fre17]{Fresse1}
B.~Fresse.
{\slshape Homotopy of operads and Grothendieck-Teichm\"uller groups. 
	Part 1: The Algebraic Theory and its Topological Background.}
Math. Surveys  Monogr.  217,
Am. Math. Soc.,  Providence, 2017.




 	\bibitem[FSV12]{fsv}
 	J. Fuchs, C. Schweigert, A. Valentino.
 {Bicategories for boundary conditions and for surface defects in 3-d TFT.}
  	 \emph{Commun. Math. Phys.} 321:543--575, 2013.

\bibitem[FS17]{fscoend}
J. Fuchs, C. Schweigert.
{Coends in conformal field theory.}
\emph{Contemp. Math.}
695:65--81, 2017.


 	








	\bibitem[Gal17]{galindo}
C. Galindo.
{Coherence for monoidal $G$-categories and braided $G$-crossed categories}.
\emph{J. Algebra} 487:118–137, 2017.




	\bibitem[GKP18]{mtrace}
N. Geer, J. Kujawa, B. Patureau-Mirand. {
M-traces in (non-unimodular) pivotal categories.} 2018.
arXiv:1809.00499 [math.RT]

		\bibitem[GNN09]{centerofgradedfusioncategories}
S. Gelaki, D. Naidu, D. Nikshych.
{Centers of Graded Fusion Categories}.
\emph{Algebra  Number Theory}, 3(8):959-990, 2009.


	\bibitem[GJ09]{goerssjardine} 
P. G. Goerss, R. F. Jardine. \emph{Simplicial Homotopy Theory.} 
Birkhäuser, 2009.


\bibitem[Hir03]{Hirschhorn}
P.~S.~Hirschhorn.
{\slshape Model categories and their localizations}.
Math.\ Surveys Monogr.\ {99}, 
Amer.\ Math.\ Soc., Providence, 2003.


\bibitem[Hoe09]{hoefel}
E. Hoefel. {OCHA and the swiss-cheese operad.} \emph{J. Homotopy Relat. Struct.} 4(1):123–151, 2009.

\bibitem[Hor17]{horel}
G. Horel.
Factorization homology and calculus à la Kontsevich Soibelman. \emph{J. Noncommutative  Geometry}    11(2):703--740, 2017.  

	 \bibitem[Hua08]{huang} Y.-Z. Huang. Rigidity and modularity of vertex tensor categories. \emph{Commun. Contemp. Math.}
10(1):871--911, 2008.

\bibitem[Idr17]{idrissi}
N. Idrissi.
{Swiss-Cheese operad and Drinfeld center.}
\emph{Isr. J. Math.}  221(2):941--972, 2017. 




\bibitem[Iv86]{iversen}
B. Iversen.
\emph{Cohomology of sheaves.}
Springer-Verlag Berlin Heidelberg, 1986.




\bibitem[KS06]{ks}
H. Kajiura, J. Stasheff. {Homotopy algebras inspired by classical open-closed string field theory.} \emph{Comm. Math. Phys.} 263(3):553--581, 2006.




\bibitem[Kas95]{kassel}
C. Kassel.
\emph{Quantum Groups}.
Springer-Verlag New York, 1995.

\bibitem[Kel99]{keller}
B. Keller.
{On the cyclic homology of exact categories}.
\emph{J. Pure Appl. Algebra} 136:1--56, 1999.




\bibitem[KL01]{kl}
T. Kerler, V. V. Lyubashenko. \emph{Non-Semisimple
	Topological Quantum Field
	Theories for 3-Manifolds
	with Corners}.
	Lecture Notes in Mathematics 1765, Springer-Verlag Berlin Heidelberg, 2001.

\bibitem[Kir04]{kirrilovg04}
A. A. Kirillov.
{On $G$-equivariant modular categories}. arXiv:math/0401119v1 [math.QA]


\bibitem[LMSS18]{svea}
S. Lentner, S. N. Mierach, C. Schweigert, Y. Sommerhäuser.
{Hochschild Cohomology and the Modular Group.}
\emph{J. Algebra}
 507:400--420, 2018.

\bibitem[LMSS20]{svea2}
S. Lentner, S. N. Mierach, C. Schweigert, Y. Sommerhäuser.
{Hochschild Cohomology, Modular Tensor Categories, and Mapping Class Groups.} arXiv:2003.06527 [math.QA]


\bibitem[LM94]{lubamajid}
V. V. Lyubashenko, S. Majid.
{Braided groups and quantum Fourier transform.} \emph{J. Algebra} 166(3):506--528, 1994.

\bibitem[Lyu95a]{lubacmp}
V. V. Lyubashenko. {Invariants of 3-manifolds and projective representations of mapping class groups  via  quantum  groups  at  roots  of  unity.} \emph{Commun. Math. Phys.} 172:467--516, 1995.


\bibitem[Lyu95b]{luba}
 V. V. Lyubashenko. {Modular transformations for tensor categories.}
\emph{J. Pure Appl. Algebra} 98(3):279--327, 1995.

\bibitem[Lyu96]{lubalex}
V. V. Lyubashenko. {Ribbon abelian categories as modular categories.}
 \emph{J. Knot Theory and its
Ramif.} 5:311--403, 1996.


\bibitem[Mac71]{maclane}
S. Mac Lane.
\emph{Categories for the Working Mathematician}. 	Graduate Texts in Mathematics 5, Springer, 1971.

\bibitem[MNS12]{maiernikolausschweigerteq}
J. Maier, T. Nikolaus, C. Schweigert.
{Equivariant Modular Categories via Dijkgraaf-Witten theory}. \emph{Adv. Theor. Math. Phys.} 16:289--358, 2012.


\bibitem[MCar94]{mcarthy}
R. McCarthy. {The cyclic homology of an exact category.} \emph{J. Pure Appl. Algebra} 93:251--296, 1994.


	\bibitem[Müg04]{mueger}
	M. Müger.
	{Galois extensions of braided tensor categories and braided crossed $G$-categories}. \emph{J. Algebra} 277:256--281, 2004.

	\bibitem[MW20a]{MuellerWoikeHH}
L. Müller, L. Woike.
{Equivariant Higher Hochschild Homology and Topological Field Theories}.
\emph{Homology Homotopy Appl.} 22(1):27--54, 2020.

\bibitem[MW20b]{littlebundles}
L. Müller, L. Woike.
{The Little Bundles Operad}.
\emph{Algebr. Geom. Topol.}
 20:2029--2070, 2020.





\bibitem[RT90]{rt1} N. Reshetikhin, V. Turaev. {Ribbon graphs and their invariants derived from
quantum groups.} \emph{Comm. Math. Phys.} 127:1--26, 1990.

\bibitem[RT91]{rt2}
N. Reshetikhin, V. Turaev. {Invariants of 3-manifolds via link polynomials and quantum groups.} \emph{Invent.
Math.} 103:547--598, 1991.

	\bibitem[Rie14]{riehl}
E. Riehl. \emph{Categorical Homotopy Theory.}
New Mathematical Monographs 24, Cambridge University Press, 2014.

	\bibitem[Rie18]{riehl18}
E. Riehl. {Homotopy coherent structures.} Lecture Notes, 2018.
arXiv:1801.07404 [math.CT]







	\bibitem[SW03]{salvatorewahl}
P. Salvatore, N. Wahl.
{Framed discs operads and Batalin-Vilkovisky algebras}. 
\emph{Quart. J. Math.} 54:213--231, 2003. 



	\bibitem[SW20]{extofk}
C. Schweigert, L. Woike.
{Extended Homotopy Quantum Field Theories and their Orbifoldization}. 
\emph{J. Pure Appl. Algebra} 
224(4), 2020.







  	







	\bibitem[Shi20]{shimizu}
K. Shimizu.
{Further results on the structure of (co)ends in finite tensor categories}. 
 	\emph{App. Cat. Structures}
 	28:237--286, 2020.
 	

\bibitem[Shi19]{shimizumodular}
K. Shimizu.
{Non-degeneracy conditions for braided finite tensor categories}. \emph{Adv. Math.} 355, 2019.

	\bibitem[Shu09]{shulman}
M. Shulman. {Homotopy Limits and Colimits and Enriched Homotopy Theory.} arXiv:math/0610194 [math.AT]






\bibitem[SZ12]{sz}
Y. Sommerh\" auser, Y. Zhu.
\emph{Hopf Algebras and Congruence Subgroups.}
Mem. Am. Math. Soc. 219(1028), Am. Math. Soc., Providence, 2012.









	\bibitem[Tur00]{turaevgcrossed}	
		V. Turaev.
		\emph{Homotopy field theory in dimension 3 and crossed group-categories.}
	arXiv:math/0005291 [math.GT]

	\bibitem[Tur10]{turaev}
V. G. Turaev. \emph{Quantum Invariants of Knots and 3-Manifolds.}
Studies in Mathematics 18, De Gruyter, 2010.

		\bibitem[Tur10-$G$]{turaevhqft}
V. Turaev.
\emph{Homotopy Quantum Field Theory}. With appendices by M. Müger and A. Virelizier. Eur. Math. Soc., 2010.

		\bibitem[TV12]{htv}
V. Turaev, A. Virelizier.
{On 3-dimensional homotopy quantum field theory, I}. 	\emph{Int. J. Math.} 23(9):1--28, 2012. 

\bibitem[TV14]{hrt}
V. Turaev, A. Virelizier.
{On 3-dimensional homotopy quantum field theory II: The surgery approach}. 
\emph{Int. J. Math.} 25(4):1--66, 2014. 


\bibitem[Vor99]{voronov}
A. A. Voronov.  \emph{The Swiss-cheese operad}.  In: Homotopy invariant algebraic structures, 239 \emph{Contemp. Math.} 365--373. Amer. Math. Soc., 1999.




	\bibitem[Wei94]{weibel}
C. A. Weibel. \emph{An introduction to homological algebra.}
Cambridge studies in advanced mathematics 38, Cambridge University Press, 1994.



\end{thebibliography}
